\documentclass[11pt]{amsart}

\usepackage{amsmath,amsthm,amssymb,amsfonts,mathrsfs,color}
\usepackage{latexsym,esint}
\usepackage{mathrsfs}  
\usepackage[colorinlistoftodos,prependcaption,textsize=tiny]{todonotes}

\usepackage[colorlinks=true,urlcolor=blue, citecolor=red,linkcolor=blue]{hyperref}
\usepackage{cleveref}
\usepackage{xcolor}
\usepackage{amsthm} 
\usepackage{latexsym,amsmath,amssymb}
\usepackage{accents}
\usepackage[colorinlistoftodos,prependcaption,textsize=tiny]{todonotes}
\usepackage{a4wide}
\usepackage{mathtools} 
\usepackage{xparse} 

\newtheorem{theorem}{Theorem}[section]
\newtheorem{proposition}{Proposition}[section]

\newtheorem{corollary}{Corollary}[section]

\def\C{{\mathcal C}}

\def \O{{\Omega}}

\def\M{{\mathcal M}}

\def\e {{\varepsilon}}
\renewcommand{\div}{{\rm div}}

\newcommand{\dd}{\,\mathrm{d}}
\newcommand{\RE}{\text{Re}}

\newcommand{\Deriv}{\mathrm {D}}
\newcommand{\GL}{\mathrm {GL}}

\def\curl{{\rm curl\,}}

\def\supp{{\rm supp\,}}

\def \p {\partial}
\def\E{\mathcal{E}}

\renewcommand{\dd}{\mathrm{d}}

\newcommand{\R}{\mathbb{R}}

\DeclareMathOperator \tr{\rm {tr}}
\DeclareMathOperator \dive{\rm {div}}
\DeclareMathOperator \ex{\text{ex}}
\DeclareMathOperator \Id{\rm {Id}}

\numberwithin{equation}{section} \numberwithin{theorem}{section}
\numberwithin{proposition}{section} \numberwithin{lemma}{section}
\numberwithin{corollary}{section}
\numberwithin{definition}{section} \numberwithin{remark}{section}

\title[Stability for limiting GL vorticities]{Stability conditions for mean-field limiting vorticities of the Ginzburg-Landau equations in 2D}

\date{\today}

\author{Rémy Rodiac}
\address[R. Rodiac]{Universit\'e Paris-Saclay, CNRS,  Laboratoire de math\'ematiques d'Orsay, 91405, Orsay, France \(\&\) Institute of Mathematics,University of Warsaw, Banacha 2, 02-097
Warszawa, Poland}

\email{remy.rodiac@universite-paris-saclay.fr, rrodiac@mimuw.edu.pl}

\begin{document}

\begin{abstract}
We analyse the limit of stable solutions to the Ginzburg-Landau (GL) equations when \(\e\), the inverse of the GL parameter, goes to zero and in a regime where the applied magnetic field is of order \(|\log \e |\) whereas the total energy is of order \(|\log \e|^2\). In order to do that we pass to the limit in the second inner variation of the GL energy. The main difficulty is to understand the convergence of quadratic terms involving derivatives of functions converging only weakly in \(H^1\). We use an assumption of convergence of energies, the limiting criticality conditions obtained by Sandier-Serfaty by passing to the limit in the first inner variation and properties of limiting vorticities to find the limit of all the desired quadratic terms. At last we investigate the limiting stability condition we have obtained. In the case with magnetic field we study an example of an admissible limiting vorticity supported on a line in a square \(\O=(-L,L)^2\) and show that if \(L\) is small enough this vorticiy satisfies the limiting stability condition whereas when \(L\) is large enough it stops verifying that condition. In the case without magnetic field we use a result of Iwaniec-Onninen to prove that every measure in \(H^{-1}(\O)\) satisfying the first order limiting criticality condition also verifies the second order limiting stability condition. 
\end{abstract}


\maketitle

\maketitle

\section{Introduction}

\subsection{The Ginzburg-Landau equations in the London limit}

The Ginzburg-Landau (GL) energy is used to describe the behaviour of type-II superconductors. In 2D, this energy can be written as
\begin{equation}\label{GL_magnetic}
\GL_\e(u,A)=\frac12 \int_\O \left(|(\nabla -iA)u|^2+\frac{1}{2\e^2}(1-|u|^2)^2\right) +\frac12 \int_{\R^2} |\curl A-h_{\ex}|^2.
\end{equation}
Here \(\O\subset \R^2\) is a smooth simply-connected bounded  domain, \(\e>0\) is a small parameter (the inverse of the GL parameter), \(h_{\ex}>0\) is another parameter representing the exterior magnetic field, \(A:=(A_1,A_2):\O\rightarrow \R^2\) is the vector-potential of the induced magnetic field which is obtained by \(h=\curl A:= \p_1A_2-\p_2 A_1\). It is sometimes more convenient to see \(A\) as a \(1\)-form \(A=A_1 dx_1+A_2 dx_2\) in \(\R^2\) and \(h\) as a \(2\)-form \(h= d A\). We will use both points of view in the following. The complex function \(u:\O\rightarrow \mathbb{C}\) is called the order parameter. The regions where \(|u| \simeq 1\) are in a superconducting phase whereas the regions where \(|u|\simeq 0\) are in a normal phase. The covariant gradient \(\nabla_A u=(\nabla-iA)u\) is a vector in \(\mathbb{C}^2\) whose coordinates are \( (\p_1^Au,\p_2^Au)=(\p_1 u-iA_1u,\p_2 u-iA_2u)\). The limit \(\e\rightarrow 0\) corresponds to extreme type-II materials and this is the regime we consider in this article. Critical points of \(\GL_\e\) in the space
\begin{equation}\label{Minimisation_space}
X:=\{ (u,A)\in H^1(\O,\mathbb{C})\times H^1_{\text{loc}}(\R^2,\mathbb{R}^2); \curl A-h_{\ex}\in L^2(\R^2)\}
\end{equation} 
are points \( (u,A)\in X\) such that 
\begin{multline}\label{eq:1st_outer_GL_magnetic}
\dd \GL_\e(u,A,v,B):= \frac{\dd }{\dd t}\Bigl|_{t=0}\GL_\e(u+tv,A+tB)=0,\\  \text{ for all} \ (v,B)\in \C^\infty(\overline{\O},\mathbb{C})\times \C^\infty_c(\R^2,\R^2).
\end{multline}
They satisfy the Euler-Lagrange equations
\begin{equation}\label{eq:GL_equations_magnetic}
\left\{
\begin{array}{rcll}
-(\nabla_A)^2u&=&\frac{u}{\e^2}(1-|u|^2) &\text{ in } \O \\
-\nabla^\perp h&=&\langle iu,\nabla_A u \rangle & \text{ in } \O \\
h&=& h_{\ex} & \text{ in }  \R^2\setminus \O \\
\nu \cdot \nabla_A u & =& 0 & \text{ on }  \p \O.
\end{array}
\right.
\end{equation}
Here the covariant Laplacian is defined by \( (\nabla_A)^2u=\p_1^A(\p_1^Au)+\p_2^A(\p_2^Au)\) and \( \langle iu,\nabla_Au\rangle\) is a vector in \(\R^2\) whose coordinates are \( (\langle iu,\p_1^Au\rangle,\langle iu,\p_2^Au\rangle )\) where,  for two complex numbers \(z,w\in \mathbb{C}\), we have denoted by \( \langle z,w\rangle \) the quantity \( \frac12 (z\bar{w}+w\bar{z})\). We also use the notation \(\nabla^\perp h=(-\p_2h,\p_1h)\) and  \(\nu\) denotes the outward unit normal to \(\p \O\).

We observe that Equations \eqref{eq:GL_equations_magnetic} and the energy \(\GL_\e\) are invariant under gauge transformations. More precisely if \((u,A)\) satisfies \eqref{eq:GL_equations_magnetic} then, for any \(f \in H^2_{\text{loc}}(\R^2,\R)\), the couple \((ue^{if},A+\nabla f)\) also satisfies \eqref{eq:GL_equations_magnetic} and  \(\GL_\e(ue^{if},A+\nabla f)=\GL_\e (u,A)\).  Physically only the gauge-invariant quantities are relevant, these are for example: \(\GL_\e\) the energy, \(|u|\) the local density of superconducting electrons pairs  (in the Barden-Cooper-Schriefer theory), \(h\) the induced magnetic field, \(j:=\langle iu,\nabla_Au\rangle \) the current vector. In order to deal with this gauge-invariance one often works in the so-called Coulomb gauge by requiring
\begin{equation}\label{eq:Coulomb_gauge}
\left\{
\begin{array}{rcll}
\dive A &=& 0 & \text{ in } \O \\
A\cdot \nu &=& 0 &\text{ on } \p \O.
\end{array}
\right.
\end{equation} 
It can be shown that if \( (u,A)\) is in \(X\), satisfies \eqref{eq:GL_equations_magnetic} and if \(A\) is in the Coulomb gauge \eqref{eq:Coulomb_gauge} then \( (u,A)\in \C^\infty(\O,\mathbb{C})\times \C^\infty(\O,\R^2)\), the bound \(|u|\leq 1\) holds and  \(h\) is in \(H^1(\O)\), see \cite[Proposition 3.8, 3.9, 3.10]{Sandier_Serfaty_2007}. Thus we can replace the fourth equation in \eqref{eq:GL_equations_magnetic} by
\begin{equation}\label{eq:Boundary_condtion}
h=h_{\ex} \text{ on } \p \O
\end{equation}
and we can replace the term \(\int_{\R^2} |\curl A-h_{\ex}|^2\) by \(\int_\O |\curl A-h_{\ex}|^2\) if we consider solutions to \eqref{eq:GL_equations_magnetic}.

The behaviour of a family of minimizers \( \{(u_\e,A_\e)\}_{\e>0}\) of \(\GL_\e\) in \(X\) in the regime \(\e\rightarrow 0\) and \(\frac{h_{\ex}}{|\log \e|}\rightarrow \lambda>0\) has been previously studied in \cite{Sandier_Serfaty_2000b, Sandier_Serfaty_2007}. The asymptotics of families of general solutions of \eqref{eq:GL_equations_magnetic} have also been investigated and can be found in \cite{Sandier_Serfaty_2003,Sandier_Serfaty_2007}. In this article we are interested in the behaviour of \textit{stable} critical points of \(\GL_\e\) in \(X\). Here \( (u,A)\) is a stable critical points of \(\GL_\e\) in \(X\) if \((u,A)\) satisfies \eqref{eq:GL_equations_magnetic} and 
\begin{multline}\label{eq:stability_condtion}
\dd^2 \GL_\e(u,A,v,B):= \frac{\dd ^2}{\dd t^2}\Bigl|_{t=0}\GL_\e(u+tv,A+tB)\geq 0, \\  \text{ for all }(v,B)\in \C^\infty(\overline{\O},\mathbb{C})\times \C^\infty_c(\R^2,\R^2).
\end{multline}
 Our aim is to understand if this stability property produces extra conditions in the limit \(\e \to 0\) compared to general critical points. Note that local minimizers are stable and thus enter the framework of our study. This question was listed as an open problem (Open problem 15) in \cite{Sandier_Serfaty_2007}. Before stating our results we recall briefly some results on global minimizers and general critical points. For \( (u_\e,A_\e)\in X\) we set 
\begin{equation}\nonumber
 j_\e:=\langle iu_\e,\nabla_{A_\e}u_\e\rangle \text{ and } \mu (u_\e,A_\e):=\curl j_\e+\curl A_\e.
\end{equation}
 
 \begin{theorem}(\cite[Theorem 7.2]{Sandier_Serfaty_2007})
 Let \( \{(u_\e,A_\e)\}_{\e>0}\) be a family of minimizers of \(\GL_\e\) in \(X\). Assume that \(h_{\ex}/|\log \e| \rightarrow \lambda\) as \(\e \rightarrow 0\) with \(0<\lambda<+\infty\), then
 \begin{equation}\nonumber
 \frac{h_\e}{h_{\ex}}\xrightharpoonup[\e\to 0]{} h_* \text{ weakly in } H^1(\O), \quad  \frac{h_\e}{h_{\ex}}\xrightarrow[\e\to 0]{} h_*
 \text{ strongly in } W^{1,p}(\O), \forall  \ 1<p<2,
 \end{equation} 
 with \(h_*\) the unique minimizer in \(\{f\in H^1_1(\O)=\{f\in H^1(\O); \tr_{| \p \O} f=1\}; \Delta f\in \mathcal{M}(\O)\}\) of 
 \begin{equation}\nonumber
 E^\lambda(f):=\frac{1}{2\lambda}\int_\O |-\Delta f+f|+\frac12 \int_\O \left( |\nabla f|^2+|f-1|^2\right),
 \end{equation}
 and the solution of the obstacle problem
 \begin{equation}\nonumber
 \left\{
 \begin{array}{rcll}
 h_*\in H^1_1(\O), \quad h_*\geq 1-\frac{\lambda}{2} \text{ in } \O \\
 \forall v\in H^1_1(\O) \text{ such that } v \geq 1-\frac{\lambda}{2}, \quad \int_\O (-\Delta h_*+h_*)(v-h_*)\geq 0.
 \end{array}
 \right.
 \end{equation}
 Furthermore
 \begin{equation}\nonumber
 \frac{\mu(u_\e,A_\e)}{h_{\ex}} \rightarrow \mu_* \text{ in } (\C^{0,\gamma}(\O))^*, \quad  -\Delta h_*+h_*=\mu_*,
 \end{equation}
 \begin{equation}\nonumber
 \lim_{\e\to 0} \frac{\GL_\e(u_\e,A_\e)}{h_{\ex}^2}=E^\lambda(h_*)=\frac{|\mu_*|(\O)}{2\lambda}+\frac12 \int_\O \left(|\nabla h_*|^2+|h_*-1|^2\right).
 \end{equation}
 \end{theorem}
This result on minimizers is actually obtained through a \(\Gamma\)-convergence result, see \cite[Chapter 7]{Sandier_Serfaty_2007}. Note that the \(\Gamma\)-limit is convex and the limiting magnetic field \(h_*\) and the limiting vorticity measure \(\mu_*\) are uniquely characterized. In particular whereas global minimizers of \(\GL_\e\) may not be unique, their vortices behave in the same way in the mean-field limit. We turn now our attention on critical points of \(\GL_\e\). We make two assumptions which were used in \cite{Sandier_Serfaty_2003} and then relaxed in \cite[Chapter 13]{Sandier_Serfaty_2007}. In all this article, unless stated otherwise we assume that \( \{(u_\e,A_\e)\}_{\e>0}\) is a family in \(X\) which satisfies
\begin{equation}\label{eq:energy_bound}
\GL_\e(u_\e,A_\e)\leq Ch_{\ex}^2
\end{equation}
\begin{equation}\label{eq:magnetic_field_bound}
\frac{h_{\ex}}{|\log \e|}\rightarrow \lambda \in (0,+\infty) \ (\text{up to a subsequence}).
\end{equation}
where \(C\) denotes a constant which is independent of \(\e\). We can then state

\begin{theorem}\label{th:critical_points}(\cite[Theorem 1]{Sandier_Serfaty_2003}, \cite[Theorem 1.7 and 13.1]{Sandier_Serfaty_2007})
Let \(\{(u_\e,A_\e)\}_{\e>0}\) be a family of points in \(X\) which solve \eqref{eq:GL_equations_magnetic} with \(A_\e\) in the Coulomb gauge \eqref{eq:Coulomb_gauge} and such that \eqref{eq:energy_bound}-\eqref{eq:magnetic_field_bound} hold. Then, up to extraction of a subsequence,
 \begin{equation}\nonumber
 \frac{h_\e}{h_{\ex}}\xrightharpoonup[\e\to 0]{} h \text{ weakly in } H^1_1(\O), \quad  \frac{h_\e}{h_{\ex}}\xrightarrow[\e\to 0]{} h
 \text{ strongly in } W^{1,p}(\O), \forall 1<p<2,
 \end{equation} 
 \begin{equation}\nonumber
 \frac{\mu(u_\e,A_\e)}{h_{\ex}} \rightharpoonup \mu \text{ in } \mathcal{M}(\O),\quad  \quad -\Delta h+h=\mu \text{ in } \O \quad h=1 \text{ on } \p \O.
 \end{equation}
 and \begin{equation}\label{eq:stress_energy}
 \dive (T_h)= 0 \text{ in } \O \quad \quad \text{ where } (T_h)_{ij}=\p_ih \p_jh -\frac12 (|\nabla h|^2+h^2)\delta_{ij}, \quad \ 1\leq i,j\leq 2 .
 \end{equation}
\end{theorem}
Here the divergence of a matrix is a vector whose components are obtained as the divergence of the rows of the matrix. It can be shown, see e.g.\ \cite[Theorem 13.1]{Sandier_Serfaty_2007}, that with the notation of the previous theorem, \(\mu(u_\e,A_\e)\) is close to a measure of the form \(2\pi \sum_{i=1}^{M_\e} d_i^\e \delta_{a_i^\e}\) where \(M_\e\in \mathbb{N}\), \(a_i^\e\) can be thought of as the center of the vortices of \(u_\e\) and \(d_i^\e\in \mathbb{Z}\) as their degrees. As noted in \cite{Sandier_Serfaty_2003}, Theorem \ref{th:critical_points} is interesting mainly for solutions such that 
\begin{equation}\label{def:N_eps}
N_\e:=\sum_{i=1}^{M_\e} |d_i^\e|
\end{equation} is of same order as \(h_{\ex}\). If this is not the case then we should look at the limit of \(\mu(u_\e,A_\e)/N_\e\) instead but we do not consider this case in this paper. The matrix (or the tensor) \(T_h\) is called the stress-energy tensor associated to the energy \(\mathcal{L}(h)=\frac12 \int_\O (|\nabla h|^2+h^2)\). Equation \eqref{eq:stress_energy} means that \(h\) is a stationary point for \(\mathcal{L}\), i.e.\ that for any  vector field \(\eta \in \C^\infty_c(\O,\R^2)\)
\begin{equation}\label{def:inner_variation_prelim}
\frac{\dd}{\dd t} \Big|_{t=0} \mathcal{L}(h_t)=0, \text{ with } h_t(x)=h(x+t\eta (x)).
\end{equation}

 This condition on \(h\) can also be viewed as a criticality condition on the limiting vorticity\footnote{In this article we always denote by \textit{limiting vorticity} a limit in the sense of measures of \( \mu(u_\e,A_\e)/h_{\ex}\).} uniquely determining \(h\) via \(\mu=-\Delta h+h\) in \(\O\) and \(h=1\) on \(\p \O\). It is obtained by passing to the limit in the first inner variations (variations of the form \eqref{def:inner_variation_prelim}) of the energy \(\GL_\e\). Although for global minimizers the limiting vorticity \(\mu_*\) is absolutely continuous with respect to the Lebesgue measure, the criticality condition \eqref{eq:stress_energy} allows for more singular measures such as measures supported on curves. It was later shown by Aydi in \cite{Aydi_2008} that some solutions of the GL equations \eqref{eq:GL_equations_magnetic} satisfying the bounds \eqref{eq:energy_bound}-\eqref{eq:magnetic_field_bound} have their vorticity measures that concentrate on lines or on circles\footnote{Solutions of some GL equations with some rotation term with vortices accumulating on curves were obtained in \cite{Aftalion_Alama_Bronsard_2005,Alama_Bronsard_2006,Alama_Bronsard_Millot_2011}. However these solutions have a number of vortices much smaller than the rotation field (the analogue of the applied field in our case). Hence, with our renormalization the limiting vorticity measure of these solutions would be \(0\) and we should divide the vorticities by another factor to have a precise description in the limit.}. The implications of the condition \eqref{eq:stress_energy} on the regularity of \(h\) and \(\mu\)  were investigated in \cite{Sandier_Serfaty_2003, Sandier_Serfaty_2007, Le_2009, Rodiac_2019}. In particular it was obtained in \cite{Rodiac_2019} that if \(h,\mu\) are as in Theorem \ref{th:critical_points} then the absolutely continuous part of \(\mu\) is equal to \(h\textbf{1}_{\{|\nabla h|=0\}}\) (\(|\nabla h|\) was shown to be a continuous function in \cite[Theorem 13.1]{Sandier_Serfaty_2007}) and the orthogonal part is supported by a  locally \(\mathcal{H}^1\) rectifiable set. Roughly speaking it says that \(\mu\) can be supported only by sets of non-zero Lebesgue measure or by some curves. 

\subsection{Main results}

In this article we investigate the following problem: does a stability condition on a family of critical points \( \{(u_\e,A_\e)\}_{\e>0}\) of \(\GL_\e\) in \(X\) imply more regularity on their limiting vorticity measures? In particular can a family of stable solutions of \eqref{eq:GL_equations_magnetic} have a limiting vorticity which concentrate on curves? To answer this question our strategy is to pass to the limit in the second inner variation of the GL energy and deduce a supplementary condition for limiting vorticity measures of stable solutions of \eqref{eq:GL_equations_magnetic}. Then we examine if this supplementary condition implies more regularity on \(\mu\).

We first explain more precisely what we mean by first and second inner variations. Let \(\eta \in \C_c^\infty(\O,\R^2)\), we consider its associated flow map \(\Phi:\R\times \O \rightarrow \R^2\) which satisfied that for every \(x\in \R^2\), the map \(t\mapsto \Phi(t,x)\) is the unique solution to 
\begin{equation}\label{def:la_coulée}
\left\{
\begin{array}{rcll}
\frac{\partial }{\partial t}\Phi(t,x)&=&\eta (\Phi(t,x)) \\
\Phi(0,x)&=&x.
\end{array}
\right.
\end{equation}
It can be seen thanks to the Cauchy-Lipschitz theory that the flow map is well-defined in \(\R \times \O\) and that it is in \(\C^\infty(\R\times \O)\). The family \(\{\Phi_t\}_{t\in \R}\) with \(\Phi_t(\cdot)=\Phi(t,\cdot)\) is a one-parameter group of \(\C^\infty\)-diffeomorphisms of \(\O\) with \(\Phi_0=\Id\). The first and second inner variations of \(\GL_\e\) at \( (u,A)\) in the direction \(\eta\) are defined by
\begin{align}
\delta \GL_\e(u,A,\eta )&=\frac{\dd}{\dd t} \Bigl|_{t=0} \GL_\e(u\circ \Phi_t^{-1}, (\Phi_t^{-1})^*A), \label{eq:1_inner_variations}\\
 \delta^2 \GL_\e(u,A,\eta)&=\frac{\dd ^2}{\dd t^2} \Bigl|_{t=0} \GL_\e(u\circ \Phi_t^{-1},(\Phi_t^{-1})^* A). \label{eq:2_inner_variations}
\end{align}
Here we have denoted by \( (\Phi_t^{-1})^* A\) the pull-back of \(A\), viewed as a 1-form, by the diffeomorphism \(\Phi_t^{-1}\). The reason for taking the pull-back \( (\Phi_t^{-1})^* A\) and not only \(A\circ \Phi_t^{-1}\) is that we need to respect the gauge invariance. This will be explained in details in Section \ref{sec:inner_variations}. Note that when working with differential forms it is customary to take inner variations as pull-backs by \(\Phi_t^{-1}\) see e.g.\ \cite{Serre_2022}. We will also see that \eqref{eq:1_inner_variations} and \eqref{eq:2_inner_variations} are well defined and give their expressions in Section \ref{sec:inner_variations}. In this paper we do not consider inner variations up to the boundary, this is because we are mainly interested in the regularity of the vorticity measures \(\mu\) in the interior. For the use of inner variations up to the boundary in different contexts we refer e.g. to \cite{Le_Sternberg_2019,Babdjian_Millot_Rodiac_2023}.

Our main result is the following.

\begin{theorem}\label{th:main1}
Let \(\{(u_\e,A_\e)\}_{\e>0}\) be a family of points in \(X\) which solve \eqref{eq:GL_equations_magnetic} with \(A_\e\) in the Coulomb gauge \eqref{eq:Coulomb_gauge} and such that \eqref{eq:energy_bound}-\eqref{eq:magnetic_field_bound} hold. Let \(h\) be the weak \(H^1\)-limit of \(h_\e/h_{\ex}\) and \(\mu\) be the limit in the sense of measure of \(\mu(u_\e,A_\e)/h_{\ex}\) given in Theorem \ref{th:critical_points}. If we assume that 
\begin{equation}\label{cond:convergence_energies}\tag{H}
\lim_{\e \to 0} \frac{\GL_\e(u_\e,A_\e)}{h_{\ex}^2}= \frac{|\mu|(\O)}{2\lambda }  +\frac12 \int_\O \left(|\nabla h|^2+|h-1|^2 \right),
\end{equation}
then, either \(h\equiv 1\) or for all \(\eta \in \C^\infty_c(\O,\R^2)\) we have
\begin{multline}\label{eq:limiting_stability}
\lim_{\e \to 0} \frac{\delta^2 \GL_\e(u_\e,A_\e,\eta)}{h_{\ex}^2}= \int_\O \left( |\Deriv \eta^T \nabla^\perp h|^2-|\nabla h|^2 \det \Deriv \eta+  h^2 [(\dive \eta)^2-\det \Deriv \eta]\right)  \\ +\frac{1}{\lambda}\int_\O \left( \frac{|\Deriv \eta|^2}{2}-\det \eta \right) \dd |\mu|=:Q_h(\eta).
\end{multline}
If in addition we assume that \(\{(u_\e,A_\e)\}_{\e>0}\) is a family of stable critical points of \(\GL_\e\) in \(X\) then \(h\) satisfies
\begin{equation}\label{eq:stability_condition_magnetic}
Q_h(\eta) \geq 0, \quad \text{ for all }\eta \in \C^\infty_c(\O,\R^2).
\end{equation} 
\end{theorem}

It was asked in \cite[Open problem 15]{Sandier_Serfaty_2007} if extra conditions such as \eqref{eq:stability_condtion} yield more regularity on the limiting vorticity measure \(\mu\). We have found that the stability condition \eqref{eq:stability_condtion} implies  \eqref{eq:stability_condition_magnetic} in the limit.  However we will also see in Proposition \ref{prop:analyze_with_magnetic} that a vorticity measure concentrated on a line can satisfy this former property. This seems to indicate that stability alone is not sufficient to imply regularity (absolute continuity with respect to the Lebesgue measure) on the limiting vorticity measure.

We now comment on our assumptions \eqref{eq:energy_bound}, \eqref{eq:magnetic_field_bound} and \eqref{cond:convergence_energies}. Assumption \eqref{eq:energy_bound} is quite natural and is satisfied by solutions constructed in \cite{Aydi_2008,Sandier_Serfaty_2007,Contreras_Serfaty_2012,Contreras_Jerrard_2022}. Assumption \eqref{eq:magnetic_field_bound} was used in \cite{Sandier_Serfaty_2003} in order to have that
\begin{equation}
N_\e\leq C h_{\ex}
\end{equation}
 where \(N_\e\) is defined in \eqref{def:N_eps}. Stable solutions of the GL equations \eqref{eq:GL_equations_magnetic} with an exterior magnetic field much larger than \(|\log \e|\) were constructed in  \cite{Sandier_Serfaty_2007,Contreras_Serfaty_2012,Contreras_Jerrard_2022} and to this respect our assumption may appear restrictive. However \eqref{eq:magnetic_field_bound} is satisfied in \cite{Aydi_2008} where solutions concentrating on lines were built. Hence singular measures can appear in the limit for this intensity of applied magnetic field and since our purpose is to study if the stability condition implies some regularity on limiting vorticity measures it seems natural to consider first this case. We also refer to \cite{Serfaty_1999CCMI,SerfatyCCMII,Serfaty_1999ARMA} and references therein for more results on stable solutions to \eqref{eq:GL_equations_magnetic}. Our main assumption \eqref{cond:convergence_energies} is satisfied by some solutions constructed in \cite{Aydi_2008}, see  Corollary 4.1 and Lemma 4.1 in \cite{Aydi_2008}. A similar assumption of convergence of energies was used in \cite{Le_2011,Le_2015,Le_Sternberg_2019} to pass to the limit in the second inner variations for the Allen-Cahn problem and also for the non-magnetic GL problem in dimension bigger than 3 and in a regime where the energy is bounded by \(C|\log \e|\). However the argument we use is quite different from the  ones in the above mentioned articles which rests upon the use of Reshetnyak's Theorem for the Allen-Cahn part or on the constancy Theorem for varifolds for the GL part. 
We note that passing to the limit in the second inner variations for these problems was later shown to be possible without the assumption of convergence of energies in \cite{Gaspar_2020} and \cite{Cheng_2020}.

\subsection{The Ginzburg-Landau equations without magnetic field}

We also consider the GL energy without magnetic field
\begin{equation}\label{eq:GL_without_magnetic_field}
E_\e(u)=\frac12 \int_\O \left(|\nabla u|^2+\frac{1}{2\e^2}(1-|u|^2)^2 \right)
\end{equation}
and the associated  Euler-Lagrange equation
\begin{equation}\label{eq:equation_GL_without}
-\Delta u=\frac{u}{\e^2}(1-|u|^2) \text{ in } \O.
\end{equation}
With a fixed boundary condition \(g \in \C^1(\p \O,\mathbb{S}^1)\) this problem has been studied by Bethuel-Brezis-H\'elein in \cite{BBH}. The asymptotic behaviour of \(\{u_\e\}_{\e>0}\), solutions to \eqref{eq:equation_GL_without}, depends on the topological degree of \(g\). For a non-zero degree of \(g\), it has been proved in \cite{BBH} that \( \{u_\e\}_{\e>0}\) converges to some limiting harmonic map with a finite number of singularities (vortices). Besides, vortices of minimizers converge to minimizers of a renormalized energy, vortices of critical points converge to critical points of this renormalized energy and the stability also passes to the limit as shown in \cite{Serfaty_2005}. Here we do not prescribe any boundary condition and we allow the number of vortices to diverge, however as in the case with magnetic field we consider only family of solutions satisfying the following bound
\begin{equation}\label{eq:energy_bound_without}
E_\e(u_\e)\leq C |\log \e|^2.
\end{equation}
A way to understand the limit as \(\e\to 0\) of solutions \(u_\e\) to \eqref{eq:equation_GL_without} is to look at their phases and at their Jacobian determinants. More precisely, since \(\O\) is simply connected and since \(\dive \langle iu_\e,\nabla u_\e \rangle =0\) in \(\O\), then, by using Poincar\'e's lemma, we can find \(U_\e\in H^1(\O,\R)\) such that 
\begin{equation}
\nabla^\perp U_\e=\langle iu_\e,\nabla u_\e\rangle  \text{ in } \O  \quad \text{ and } \quad \int_\O U_\e=0.
\end{equation} 
Note that \(\p_1 u_\e \wedge \p_2 u_\e= \curl \langle iu_\e,\nabla u_\e \rangle =\Delta U_\e\). Hence the Laplacian of \(U_\e\) is the Jacobian determinant of \(u_\e\) and this quantity was proved to play a prominent role in \cite{Jerrard_Soner_2002b}. We can see here the analogy between \(U_\e\) and the magnetic field \(h_\e\) in the full GL model. The \(\Gamma\)-limit of \(E_\e\) in the regime \(E_\e(u_\e) \simeq |\log \e|^2\) has been studied in \cite{Jerrard_Soner_2002} where results analogous to the ones in \cite{Sandier_Serfaty_2000b} are obtained. In particular, in that case, the \(\Gamma\)-limit of \(E_\e/|\log \e|^2\) is given by \( \frac{|\mu|(\O)}{2}+\frac12 \int_\O   |\nabla U|^2\) where \(U\) is the weak limit in \(H^1\) of \(U_\e/|\log \e|\) and \(\mu \in \mathcal{M}(\O)\) is the limit in the sense of measures of \(\Delta U_\e/|\log \e|\) (up to extraction). For solutions to \eqref{eq:equation_GL_without}, the following results were obtained in \cite{Sandier_Serfaty_2003, Sandier_Serfaty_2007}: the limit satisfies \(\Delta U=\mu \in H^{-1}(\O)\) and the stress-energy tensor \((S_U)_{ij}=2\p_iU\p_jU-|\nabla U|^2 \delta_{ij}\) for \(1\leq i,j\leq 2\) is divergence-free in \(\O\). A  way to reformulate this property is to say that the quantity \((\p_xU)^2-(\p_yU)^2-2i(\p_xU)(\p_yU)\) is homomorphic in \(\O\). Using complex analysis techniques and techniques from sets of finite perimeters it was proved in \cite{Rodiac_2016} that if \(\mu\) satisfies the previous limiting critical conditions then \(\mu\) is supported on a rectifiable set which is locally the the zero set of a harmonic function. Hence \(\mu\) cannot be absolutely continuous with respect to the Lebesgue measure unless \(\mu=0\). We consider here stable solutions of \eqref{eq:equation_GL_without}, i.e.\ solutions satisfying
\begin{equation}\label{eq:stability_GL_without}
\frac{\dd ^2}{\dd t^2}\Big|_{t=0} E_\e(u_\e+tv)\geq 0 \quad \forall v \in \C^\infty_c(\O,\mathbb{C}).
\end{equation}
As previously, we define the first and second inner variations of \(E_\e\) with respect to a \(\eta\in \C^\infty_c(\O, \R^2)\) by 
\begin{align}
\delta E_\e(u,\eta )=\frac{\dd}{\dd t} \Bigl|_{t=0} E_\e(u\circ \Phi_t^{-1}),  \quad  \quad 
 \delta^2 E_\e(u,A,\eta)=\frac{\dd ^2}{\dd t^2} \Bigl|_{t=0} E_\e(u\circ \Phi_t^{-1}), \label{eq:2_inner_variations_without}
\end{align}
where \(\Phi_t\) is defined in \eqref{def:la_coulée}. We will prove in Section \ref{sec:inner_variations} that these quantities are well-defined.

\begin{theorem}\label{th:main_2_without}
Let \( \{u_\e\}_{\e>0}\) be a family of solutions to \eqref{eq:equation_GL_without} satisfying \eqref{eq:stability_GL_without} and \(E_\e(u_\e)\leq C|\log \e|^2\). As shown in \cite[Theorem 3]{Sandier_Serfaty_2003} or \cite{Sandier_Serfaty_2007}, up to a subsequence, \(U_\e/|\log \e|\rightharpoonup U\) in \(H^1(\O)\) and \(\Delta U_\e / |\log \e|=\curl \langle iu_\e,\nabla u_\e \rangle/|\log \e|  \rightharpoonup \mu \in \mathcal{M}(\O)\) with
\begin{equation}\label{eq:critical_conditions_without}
\Delta U=\mu \text{ and } (\p_xU)^2-(\p_yU)^2-2i(\p_xU)(\p_yU) \text{  is homomorphic in } \O.
\end{equation}
If we assume  that 
\begin{equation}\label{cond:conv_energy_2}\tag{H'}
\lim_{\e \to 0} \frac{E_\e(u_\e)}{|\log \e|^2}=\frac{|\mu|(\O)}{2}+\frac12 \int_\O  |\nabla U|^2 ,
\end{equation}
 then for all \(\eta \in \C^\infty_c(\O,\R^2)\) we have
\begin{equation}\label{eq:stability_limit_without}
\lim_{\e \to 0} \frac{\delta^2 E_\e (u_\e,\eta)}{|\log \e|^2}=\frac12 \int_\O |\Deriv \eta ^T\nabla^\perp U|^2-|\nabla U|^2 \det \Deriv \eta+ \int_\O \left(\frac{|\Deriv \eta|^2}{2}-\det \Deriv \eta \right) \dd |\mu|\\
=:\tilde{Q}_U (\eta).
\end{equation}
\end{theorem}

 Again we can see that solutions to \eqref{eq:equation_GL_without} satisfying \eqref{eq:stability_GL_without} have the property that, in the limit, \(\tilde{Q}_U(\eta)\geq 0\) for all admissible \(\eta\). We can also ask if that condition provides more regularity on the possible limiting vorticies. Here this is never the case. Indeed, thanks to a recent result of Iwaniec-Onninen \cite[Theorem 1.12]{Iwaniec_Onninen_2022}, we are able to prove that every measure satisfying the limiting criticality conditions \eqref{eq:critical_conditions_without} also satisfies the limiting stability condition: \(\tilde{Q}_U(\eta)\geq 0\) for all admissible \(\eta\), see Proposition \ref{prop:No_regularity_without}. We should observe that, contrarily to the case of the GL equations with magnetic field, it is still an open problem to determine if there exist solutions to \eqref{eq:equation_GL_without} with a diverging number of vortices such that their limiting vorticities concentrate on curves (which should be locally the zero set of some, possibly multi-valued, harmonic functions according to \cite{Rodiac_2016}).

\subsection{Method of proof}
 
 For smooth critical points of energies, inner variations are strongly related to outer variations which are defined, in the case of \(\GL_\e\), for \((u,A)\in X\) and \((v,B)\in \C^\infty(\overline{\O})\times \C^\infty_c(\R^2) \) by \eqref{eq:1st_outer_GL_magnetic} and \eqref{eq:stability_condtion}.
Although  outer variations are of more common use in variational problems, it has been observed that inner variations are useful to understand the limit of singularly perturbed problems such as  the Allen-Cahn (AC) problem or the GL problem since these are variations which do move the singularities, see e.g.\ \cite{BBH,Hutchinson_Tonegawa_2000,Lin_Riviere_2001,Bethuel_Brezis_Orlandi_2001,Sandier_2001,Sandier_Serfaty_2003}. 

 More recently some interest has grown in understanding how the stability condition passes to the limit in the above mentioned problems, we refer for example to \cite{Le_2011,Le_2015,Le_Sternberg_2019,Gaspar_2020,Cheng_2020}. Again it turns out that studying the second inner variations are more appropriate to understand the limiting behaviour of stable solutions to the AC or GL problems. Looking at the expressions given by the first and second inner variations of GL type functionals, see Proposition \ref{prop:inner_variations_GL_magnetic} for the formulas, one can see that one of the difficulty is to pass to the limit in quadratic expressions involving derivatives of the unknown functions whereas only weak convergence in \(H^1\) of these functions is available. For example the vanishing of the first inner variation of \(\GL_\e\) provides
\begin{multline}\label{eq:tensor_GL_magnetic}
\dive( T_\e)=0 \text{ in } \O, \\
 \quad \text{ with }(T_\e)_{ij} = \langle \p_i^{A_\e} u_\e,\p_j^{A_\e}u_\e\rangle -\frac12\left( |\nabla_{A_\e}u_\e|^2+\frac{1}{2\e^2}(1-|u_\e|^2)^2-h_\e^2 \right)\delta_{ij}.
\end{multline}
 The formula for the second inner variation is given in Proposition \ref{prop:inner_variations_GL_magnetic}. 
Let us briefly recall how Sandier-Serfaty in \cite{Sandier_Serfaty_2003,Sandier_Serfaty_2007}  managed to pass to the limit in \eqref{eq:tensor_GL_magnetic}. First we can see, at least formally, that 
\begin{equation}
\frac{T_\e}{h_{\ex}^2}\simeq L_\e \text{ where } L_\e=\frac{1}{h_{\ex}^2}\left( -\p_ih_\e\p_jh_\e +\frac12(|\nabla h_\e|^2+h_\e^2)\right).
\end{equation}
Although \(h_\e/h_{\ex}\) converges only weakly in \(H^1\), Sandier-Serfaty succeeded in passing to the limit in the equation \(\dive(T_\e)=0\) by showing that the convergence of \(h_\e/h_{\ex}\) is actually strong in \(H^1\) outside a set of arbitrary small perimeter and by using  the equation along with a co-area formula argument. This type of problem has the same flavour of the problem of understanding the limit of solutions to the incompressible Euler equations in 2D fluid mechanics see \cite{DiPerna_Majda_1988,Delort_1991}.

To pass to the limit in the second inner variation we cannot use the same argument since we have to pass to the limit in an inequality and not in an equality. We must then understand the limit of all the quadratic terms appearing in the formula given in Proposition \ref{prop:inner_variations_GL_magnetic}. Assumption \eqref{cond:convergence_energies} allows us to show that the potential term \(\frac{1}{2\e^2 h_{\ex}^2}(1-|u|^2)^2\) converges strongly towards zero in \(L^1(\O)\).  Next we say that  \(|\p_1^{A_\e}u_\e|^2/h_{\ex}^2\rightharpoonup |\p_2h|^2+\nu_1, \ |\p_2^{A_\e}u_\e|^2/h_{\ex}^2\rightharpoonup |\p_1h|^2+\nu_2\) and \(\langle \p_1^{A_\e}u_\e,\p_2^{A_\e}u_\e\rangle /h_{\ex}^2\rightharpoonup -\p_1 h,\p_2h+\nu_3\) where \(\nu_1,\nu_2,\nu_3\) are Radon measures in \(\O\) and the convergence takes place in the sense of measures. We can then pass to the limit in the equation \(\dive ( T_\e/h_{\ex}^2)=0\) and use Theorem \ref{th:critical_points} to deduce an equation on \(\nu_1,\nu_2,\nu_3\) in the interior \(\O\). This equation actually means  that \(\nu_1-\nu_2-i\nu_3\) is holomorphic in \(\O\). We use again assumption \eqref{cond:convergence_energies}, along with the  description of possible limiting vorticity measures \(\mu\) obtained in \cite{Rodiac_2019} to obtain that \(\nu_1=\nu_2=\mu/2\lambda\) and \(\nu_3=0\) on a ball contained in \(\O\) if \(h\) is not constantly equal to \(1\).  Then the principle of isolated zeros gives \(\nu_1=\nu_2=|\mu|/\lambda\) and \(\nu_3=0\) in all \(\O\). Finally we analyse the inequality obtained by passing in the limit in the second inner variation. In the case with magnetic field we show on one example that  we can have \(Q_h(\eta)\geq 0\), where \(Q_h\) is defined in \eqref{eq:limiting_stability}, for all \(\eta \in \C^\infty_c(\O, \R^2)\) and \(\mu\) supported by a line. We use similar arguments to treat the case without magnetic field to pass to the limit in the second inner variation. We then employ a result of Iwaniec-Onninen \cite{Iwaniec_Onninen_2022} to obtain that \(\tilde{Q}_U(\eta)\geq 0\) for all \(\eta \in \C^\infty_c(\O,\R^2)\), with \(\tilde{Q}_U\) defined in \eqref{eq:stability_limit_without}.

\subsection{Organization of the paper}

The paper is organized as follows. In Section \ref{sec:inner_variations} we compute the expressions of the first and second inner variations. We also explain the link between inner and  outer variations. Section \ref{sec:pass_limit} is dedicated to show how to pass to the limit in the second inner variation. In order to do this we study the limit of all the quadratic terms appearing in the second inner variations of \(\GL_\e\) by using an argument of defect measures and by using the limit of the first inner variation. Finally Section \ref{sec:Analysis_limiting_condition} is devoted to analyse the limiting stability condition obtained in Theorem \ref{th:main1} and Theorem \ref{th:main_2_without}.

\subsection{Notations}
For \(u,v\) two vectors in \(\R^2\) we denote by \(u\cdot v\) their inner product. When \(u,v\) are identified with complex numbers then we denote also their inner products by \(\langle u, v \rangle\). If \(\eta \in \C^\infty(\O,\R^2)\) is a smooth vector field we use \(\Deriv \eta\) to denote its differential. When we apply this differential to a vector \(x\in \R^2\) we use \(\Deriv \eta.x\). The second derivative of a smooth vector field  \(\eta\) applied to two vectors \(x,y\in \R^2\) is denoted by \( \Deriv^2 \eta [x,y]\). For two matrices \(M,N \in \mathcal{M}_2(\R)\) we let \( M:N:=\tr (M^TN)\) denote their inner product and \(\|M\|\) the associated norm, with \(M^T\) the transpose matrix of \(M\). For two vectors \(x,y\in \R^2\) we define their tensor products to be a matrix in \(\mathcal{M}_2(\R)\) whose entries are given by \( (x \otimes y)_{ij}=x_iy_j\). Note that we have the relation \(M.x \cdot y=M: y\otimes x\). For \(0\) and \(1\)-forms \(f\) and \(A\) we denote by \(df\) and \( dA\) their exterior derivatives. For a function \(h\) regular enough we set \(\nabla^\perp h= (-\p_2h,\p_1h)^T\). For a Radon measure \(\mu \in \mathcal{M}(\O)\) we denote by \(|\mu|(\O)\) its total variation. When we need to evaluate the energy on a subdomain \(V\subset \O\) we write \(\GL_\e(u,A,V)\).

\medskip

\textbf{Acknowledgements:} I would like to thank Etienne Sandier for useful discussions about this topic. This research is part of the project No. 2021/43/P/ST1/01501 co-funded by the National Science Centre and the European Union Framework Programme for Research and Innovation Horizon 2020 under the Marie Skłodowska-Curie grant agreement No. 945339. For the purpose of Open Access, the author has applied a CC-BY public copyright licence to any Author Accepted Manuscript (AAM) version arising from this submission.

\section{Inner variations}\label{sec:inner_variations}

In this section we compute the first and second inner variations of \(\GL_\e\) defined in \eqref{eq:1_inner_variations}-\eqref{eq:2_inner_variations} and we explain the link with the outer variations \eqref{eq:1st_outer_GL_magnetic}-\eqref{eq:stability_condtion}.

\subsection{Variations and gauge invariance}

Since the functional \(\GL_\e\) and physical quantities are gauge invariant, we should  use variations for which the notion of stationarity does not depend on the gauge. That is why we have defined inner variations as \( (u\circ \Phi_t^{-1}, (\Phi_t^{-1})^* A)\) and not simply as \( (u \circ \Phi_t^{-1}, A\circ \Phi_t^{-1})\). 

\begin{proposition}
Let \( (u,A)\) and \((\tilde{u},\tilde{A})\) be in \(X\) such that there exists \(f\in H^2_{\text{loc}}(\R^2,\R)\) with \( \tilde{u}=ue^{if}\) and \( \tilde{A}= A+ \nabla f\). Then for any \(\eta \in \C^\infty_c(\O,\R^2)\) 
\begin{equation}\label{eq:gauge_invariance_inner_var}
\delta \GL_\e(u,A,\eta )=\delta \GL_\e (\tilde{u},\tilde{A},\eta ) \quad \text{ and } \quad  \delta^2 \GL_\e(u,A,\eta )=\delta^2\GL_\e (\tilde{u},\tilde{A},\eta),
\end{equation}
with \(\delta \GL_\e(u,A,\eta)\) and \( \delta^2 \GL_\e(u,A,\eta )\) defined in \eqref{eq:1_inner_variations}-\eqref{eq:2_inner_variations}.
\end{proposition}

\begin{proof}
We let \( (u_t,A_t):=(u\circ \Phi_t^{-1}, (\Phi_t^{-1})^* A)\), where \(\Phi_t\) is defined in \eqref{def:la_coulée}. The gauge invariance implies that
\begin{align*}
\GL_\e (u_t,A_t)&=\GL_\e(u_te^{if},A_t+d f) \\
&=\GL_\e(u_te^{if_t} e^{i(f-f_t)},A_t+df_t+d(f-f_t))
\end{align*}
where \(f_t=f\circ \Phi_t^{-1}\). Since  \( \Deriv f_t =\Deriv f(\Phi_t^{-1}).\Deriv \Phi_t^{-1}\) we find that, as forms,  \(df_t=(\Phi_t^{-1})^*df\). Hence we infer that \( A_t+df_t= (\Phi_t^{-1})^*(A+df)\) and thus, using once again the gauge invariance,
\[ \GL_\e (u_t,A_t)=\GL_\e(\tilde{u}_t,\tilde{A}_t)\]
with \((\tilde{u}_t,\tilde{A}_t)=(\tilde{u}\circ \Phi_t^{-1}, (\Phi_t^{-1})^* \tilde{A})\). Differentiating with respect to \(t\) yields \eqref{eq:gauge_invariance_inner_var}.
\end{proof}

It can be checked by direct computation that the quantity \(\frac{\dd}{\dd t}\Big |_{t=0} \GL_\e(u\circ \Phi_t^{-1},A\circ \Phi_t^{-1})\) and its second order analogue are not gauge invariant. However we observe that outer variations are also well-adapted to the gauge invariance in the sense that if \((u,A)\in X\) is a critical point of \( \GL_\e\) then \( (ue^{if},A+df)\) is also a critical point of \(\GL_\e\) in \(X\) for \(f\in H^2_{\text{loc}}(\R^2,\R)\) and if \( (u,A)\) is stable then so is \((ue^{if},A+df)\). This follows for example by observing that for \(t\in \R\) and for any \( (v,B)\in X\) we have
\(\GL_\e(u+tv,A+tB)=\GL_\e (ue^{if}+tve^{if},A+df+tB)\). Hence differentiating with respect to \(t\) entails that \( \dd \GL_\e(u,A).(v,B)=\dd \GL_\e(ue^{if},A+df).(ve^{if},B)\) and \( \dd^2\GL_\e (u,A).(v,B)=\dd^2\GL_\e(ue^{if},A+df).(ve^{if},B)\).

\subsection{Inner variations and outer variations for the GL energy}

To compute the first and second inner and outer variations of the GL energy in the magnetic and non-magnetic case we first rewrite these energies by using the vectorial setting instead of the complex one. Namely, we see the order parameter as a map \(u:\O\rightarrow \R^2\) and we write \(\Deriv u\in \mathcal{M}_2(\R)\) for its differential (instead of \(\nabla u\) for its complex gradient). We can check that the complex covariant gradient \( (\nabla-iA)u\) corresponds to the real matrix \(\begin{pmatrix}
\p_1u_1+A_1u_2 & \p_2u_1+A_2u_2 \\
\p_1u_2-A_1u_1 & \p_2u_2-A_2u_1
\end{pmatrix}\). Thus if we define \(u^\perp:= \begin{pmatrix}
 -u_2 \\
 u_1
\end{pmatrix}\) we find that \( (\nabla-iA)u\) corresponds to \( \Deriv u-u^\perp A^T\) and 
\begin{equation}
\GL_\e(u,A)=\frac12 \int_\O \left( |\Deriv u-u^\perp A^T|^2+\frac{1}{2\e^2}(1-|u|^2)^2\right) +\int_{\R^2}|\curl A-h_{\ex}|^2.
\end{equation}

General formulas for the first and inner variations of functionals are given in \cite{Le_2011,Le_2015,Le_Sternberg_2019}. We present the computations here because our setting is slightly different due to the presence of the magnetic field and the term \((\Phi_t^{-1})^*\circ A\).

\begin{proposition}\label{prop:inner_variations_GL_magnetic}
Let \(\eta \in \C^\infty_c(\O,\R^2)\),  \(\zeta:=\Deriv \eta .\eta\) and \((u,A)\in X\). Then, with definitions \eqref{eq:1_inner_variations} and \eqref{eq:2_inner_variations}, we have
\begin{multline}\nonumber
\delta \GL_\e(u,A,\eta )= \int_\O \Bigl[ \frac12 \left( |\Deriv u-u^\perp A^T|^2-h^2+\frac{1}{2\e^2}(1-|u|^2)^2 \right)\Id \\-(\Deriv u-u^\perp A^T)^T(\Deriv u-u^\perp A^T) \Bigr]: \Deriv \eta
\end{multline}
\begin{multline}\nonumber
\delta^2\GL_\e(u,A,\eta)=\delta \GL_\e(u,A, \zeta) +\int_\O \Bigl[ |(\Deriv u-u^\perp A^T)\Deriv \eta|^2 -|\Deriv u-u^\perp A^T|^2 \det \Deriv \eta \\
+h^2( (\div \eta)^2-\det \Deriv \eta) +\frac{1}{2\e^2}(1-|u|^2)^2\det \Deriv \eta \Bigr].
\end{multline}
\end{proposition}

\begin{proof}
Let \(\{\Phi_t\}_{t\in \R}\) be the flow associated to \(\eta \in \C^\infty_c(\O,\R^2)\) defined in \eqref{def:la_coulée}, and let \( (u_t,A_t):=(u\circ \Phi_t^{-1}, (\Phi_t^{-1})^*A)\). By definition of the pull-back, 
\begin{align}
A_t&=A_1\circ \Phi_t^{-1} d(\Phi_t^{-1})_1+A_2\circ \Phi_t^{-1} d(\Phi_t^{-1})_2 \nonumber \\
&=(A\circ \Phi_t^{-1})\cdot \p_1 (\Phi_t^{-1}) dx_1 +(A\circ \Phi_t^{-1})\cdot \p_2 (\Phi_t^{-1}) dx_2. \label{eq:DL_pull_back}
\end{align}
By identifying \(A_t\) with a vector field in \(\R^2\) we find that \(A_t= \Deriv \Phi_t^{-T}. (A\circ \Phi_t^{-1}) \). Thus \( \Deriv u_t-u_t^\perp A_t^T= \left[(\Deriv u-u^\perp A^T)\circ \Phi_t^{-1} \right]\Deriv \Phi_t^{-1}\) and by using the change of variables \(x=\Phi_t(y)\) we find
\begin{align}
\int_\O |  \Deriv u_t-u_t^\perp A_t^T|^2&=\int_\O |(\Deriv u-u^\perp A^T)(\Phi_t^{-1}(x))\Deriv \Phi_t^{-1}(x)|^2 \dd x \nonumber \\
&= \int_\O |(\Deriv u-u^\perp A^T)(y) \Deriv \Phi_t^{-1}(\Phi_t(y))|^2 \det \Deriv \Phi_t(y) \dd y \nonumber \\
&=\int_\O |(\Deriv u-u^\perp A^T)(y) (\Deriv \Phi_t(y))^{-1}|^2 \det \Deriv \Phi_t(y) \dd y. \nonumber
\end{align}
We now look for an expansion of \((\Deriv \Phi_t)^{-1}\) and \(\det \Deriv \Phi_t\). We use the Taylor formula with integral remainder  and equation \eqref{def:la_coulée} to say that 

\begin{align}
\Phi_t(x)&=x+t\p_t |_{t=0} \Phi_t(x)+\frac{t^2}{2} \p^2_{tt}|_{t=0} \Phi_t(x)+O(t^3) \nonumber \\
&= x+t\eta (\Phi_t(x))+\frac{t^2}{2}\Deriv \eta (x).\eta(x) +O(t^3) \label{eq:expansion_Phi}
\end{align}
where, thanks to the compactness of the support of \(\eta\) the term \( O(t^3)\) is such that \(O(t^3)/t^3\) is bounded uniformly in \(x\in \O\). We can check that we can differentiate with respect to \(x\) under the integral sign giving the term \(O(t^3)\) to obtain that  \( \Deriv \Phi_t= \Id+t \Deriv \eta+\frac{t^2}{2}\Deriv \zeta+O(t^3)\) with \(\zeta= \Deriv \eta.\eta\). Now we use that for a matrix \(M\in \mathcal{M}_2(\R)\) such that \(\|M\|<1\) we have \( (I+M)^{-1}=I-M+M^2 +O(\|M\|^3)\) to conclude that
\begin{align}
(\Deriv \Phi_t)^{-1}=\Id-t\Deriv \eta-\frac{t^2}{2}\Deriv \zeta +t^2(\Deriv \eta)^2 +O(t^3). \nonumber
\end{align}
To compute the determinant \(\det \Deriv \Phi_t\) we recall that for two matrices \(M,N\) we have
\begin{equation}\label{eq:DL_det}
\det \left( \Id+tM+\frac{t^2}{2}N\right)=1+t\tr(M)+\frac{t^2}{2}\left[\tr(N)+(\tr(M))^2-\tr(M^2)) \right] +O(t^3)
\end{equation}
and that
 \begin{align*}
(\tr (\Deriv \eta))^2-\tr (\Deriv \eta)^2&=(\dive \eta)^2-\tr (\Deriv \eta)^2 =2\det \Deriv \eta
\end{align*}
since
\begin{align}\label{eq:matrix_identity}
\Deriv \eta \left[(\dive \eta)\Id-\Deriv \eta\right]&=\begin{pmatrix}
\p_1\eta_1 &\p_2\eta_1 \\
\p_1\eta_2 &\p_2\eta_2
\end{pmatrix}\begin{pmatrix}
\p_2\eta_2 &-\p_2\eta_1 \\
-\p_1\eta_2 &\p_1\eta_1
\end{pmatrix} \nonumber\\
&=\begin{pmatrix}
\p_1\eta_1\p_2\eta_2-\p_2\eta_1\p_1\eta_2 &0 \\
0 &-\p_2\eta_1\p_1\eta_2+\p_1\eta_1\p_2\eta_2\end{pmatrix} \nonumber\\
&=(\det \Deriv \eta) \Id.
\end{align}
Thus
\begin{equation}\label{eq:expansion_det_concr}
\det \Deriv \Phi_t = 1+t \dive \eta+\frac{t^2}{2}\dive \zeta+t^2 \det \Deriv \eta +O(t^3).
\end{equation}
Hence we expand
\begin{align}
\int_\O |  \Deriv u_t-u_t^\perp A_t^T|^2&=\int_\O \Biggr[\left| (\Deriv u-u^\perp A^T)\left(\Id-t\Deriv \eta -\frac{t^2}{2}\Deriv \zeta +t^2 ( \Deriv \eta)^2 +O(t^3)\right)\right|^2 \nonumber \\
& \quad \quad \times \left(1+t\dive \eta +\frac{t^2}{2}\div \zeta+t^2 \det \Deriv \eta +O(t^3)\right) \Biggl]  \nonumber \\
&= \int_\O \Bigl[ |\Deriv u-u^\perp A^T|^2-2t (\Deriv u-u^\perp A^T):(\Deriv u-u^\perp A^T) \Deriv \eta \nonumber \\
& \quad + t|\Deriv u-u^\perp A^T|^2 \dive \eta \nonumber \\
& \quad \quad  -t^2 (\Deriv u-u^\perp A^T):(\Deriv u-u^\perp A^T) \Deriv \zeta+\frac{t^2}{2}|\Deriv u-u^\perp A^T|^2 \div \zeta \nonumber \\
& \quad \quad \quad +t^2|(\Deriv u-u^\perp A^T)\Deriv \eta|^2+2t^2 (\Deriv u-u^\perp A^T):(\Deriv u-u^\perp A^T) (\Deriv \eta)^2 \nonumber \\
&\quad \quad \quad \quad -2t^2 (\Deriv u-u^\perp A^T):(\Deriv u-u^\perp A^T)\Deriv \eta \dive \eta \nonumber \\
&\quad \quad \quad \quad \quad +t^2|\Deriv u-u^\perp A^T|^2\det \Deriv \eta +O(t^3)\Bigr] \nonumber \\
&=\int_\O \Bigl[ |\Deriv u-u^\perp A^T|^2-2t (\Deriv u-u^\perp A^T):(\Deriv u-u^\perp A^T) \Deriv \eta \nonumber \\
& \quad + t|\Deriv u-u^\perp A^T|^2 \dive \eta \nonumber \\
& \quad \quad  -t^2 (\Deriv u-u^\perp A^T):(\Deriv u-u^\perp A^T) \Deriv \zeta+\frac{t^2}{2}|\Deriv u-u^\perp A^T|^2 \div \zeta \nonumber \\
& \quad \quad \quad +t^2 |(\Deriv u-u^\perp A^T)\Deriv \eta|^2-|\Deriv u-u^\perp A^T|^2\det \Deriv \eta\Big] +O(t^3), \label{eq:inner_magnetic_grad}
\end{align}
where we have used \eqref{eq:matrix_identity} again. On the other hand, we know that \begin{align*}
h_t&:= d A_t= d \left[ (\Phi_t^{-1})^*A\right]=(\Phi_t^{-1})^*dA =(h\circ \Phi_t^{-1} ) (\det D\Phi_t^{-1}) dx_1\wedge dx_2.
\end{align*}  Hence 
\begin{align}
\int_\O |h_t-h_{\ex}|^2 &=\int_\O |h(\Phi_t^{-1}(x))\det \Deriv \Phi_t^{-1}(x)-h_{\ex}|^2 \dd x \nonumber\\
&=\int_\O |h(y)\det \Deriv \Phi_t^{-1}(\Phi_t(y))-h_{\ex}|^2 \det \Deriv \Phi_t(y) \dd y \nonumber \\
&=\int_\O | h(y)\det (\Deriv \Phi_t(y))^{-1}-h_{\ex}|^2 \det \Deriv \Phi_t(y) \dd y. \nonumber 
\end{align}
By using \eqref{eq:expansion_det_concr} we find that 
\begin{align}
\int_\O |h_t-h_{\ex}|^2 &=\int_\O \Bigl[\left|h\left(1-t\div \eta+\frac{t^2}{2}\div \zeta-t^2 \det \Deriv \eta+t^2 (\dive \eta)^2\right)-h_{\ex}\right|^2 \nonumber  \\
&\quad \quad \times \left(1+t\dive \eta +\frac{t^2}{2}\dive \zeta+t^2\det \Deriv \eta\right)\Bigr]  +O(t^3) \nonumber \\
&= \int_\O \Bigl[|h-h_{\ex}|^2 -2t(h-h_{\ex})h  \dive \eta -t^2(h-h_{\ex})h\dive \zeta +t^2 h^2 (\dive \eta)^2 \nonumber \\
& \quad -2t^2 (h-h_{\ex})h\det \Deriv \eta +2t^2(h-h_{\ex})h(\dive \eta )^2 +t|h-h_{\ex}|^2\dive \eta \nonumber \\
& \quad \quad + \frac{t^2}{2}|h-h_{\ex}|^2\dive \zeta +t^2|h-h_{\ex}|^2\det \Deriv \eta -2t^2 (h-h_{\ex})h(\dive \eta)^2 \Bigr] + O(t^3) \nonumber \\
&= \int_\O  \Bigl[|h-h_{\ex}|^2-th^2\dive \eta +th_{\ex}^2\dive \eta-t^2h^2\dive \eta +t^2h_{\ex}^2\dive \eta \nonumber \\
&\phantom{aaa} -t^2h^2\det \Deriv \eta +h_{\ex}^2\det \Deriv \eta +t^2h^2 (\dive \eta)^2\Bigr] +O(t^3) \nonumber\\
&=\int_\O \Bigl[|h-h_{\ex}|^2-th^2\dive \eta-t^2h^2\dive \zeta -t^2h^2\det \Deriv \eta +t^2h^2 (\dive \eta)^2 \Bigr] +O(t^3). \label{eq:inner_h} 
\end{align}
We have used that, since \(\eta\) has compact support, \(\int_\O \dive \eta\) and \(\int_\O \det \Deriv \eta=\frac12 \int_\O \dive (\eta\wedge \p_2 \eta, \p_1 \eta \wedge \eta)\) vanish. At last, using again \eqref{eq:expansion_det_concr}, we compute
\begin{align}
\int_\O (1-|u_t|^2)^2&= \int_\O (1-u(\Phi_t^{-1}(x))|^2)^2 \dd x= \int_\O (1-|u(y)|^2)^2 \det \Deriv \Phi_t(y) \dd y \nonumber \\
&=\int_\O (1-|u|^2)^2 (1+t \dive \eta +\frac{t^2}{2}\dive \zeta +t^2 \det \Deriv \eta)+O(t^3). \label{eq:inner_potential}
\end{align}
Putting together \eqref{eq:inner_magnetic_grad}, \eqref{eq:inner_h} and \eqref{eq:inner_potential} yields the result. 
\end{proof}

Similar but simpler computations give

\begin{proposition}
Let \(\eta \in \C^\infty_c(\O,\R^2)\) and let \((u,A)\in X\) then, with definitions \eqref{eq:1_inner_variations} and \eqref{eq:2_inner_variations} we have
\begin{equation}\nonumber
\delta E_\e(u,\eta)=\int_\O \Bigl[ \frac12 \left(|\Deriv u|^2+\frac{1}{2\e^2}(1-|u|^2)^2 \right) \dive \eta -(\Deriv u)^T \Deriv u:\Deriv \eta\Bigr] ,
\end{equation}
\begin{equation}\nonumber
\delta^2E_\e(u,\eta)=\int_\O |\Deriv u \Deriv \eta|^2-|\Deriv u|^2\det \Deriv \eta +\frac{1}{2\e^2}(1-|u|^2)^2\det \Deriv \eta.
\end{equation}
\end{proposition}
It can be seen that \(\GL_\e\) is infinitely  G\^ateaux-differentiable on \(X\) and its first and second variations are given in the following proposition.

\begin{proposition}\label{prop:outer_variations_GL}
The first and second outer variation of \(\GL_\e\) at \((u,A)\) with respect to \((v,B)\in \C^\infty(\overline{\O},\R^2)\times \C^\infty_c(\R^2,\R^2)\), defined in \eqref{eq:1st_outer_GL_magnetic}-\eqref{eq:stability_condtion}
are given by
\begin{align*}
\dd \GL_\e (u,A,v,B) &= \int_\O (\Deriv u-u^\perp A^T) :(\Deriv v-v^\perp A^T)-(\Deriv u-u^\perp A^T):u^\perp B^T \\
& \quad  +(h-h_{\ex})\curl B-\frac{1}{\e^2}(1-|u|^2)u\cdot v
\end{align*}
\begin{align*}
\dd^2\GL_\e (u,A,v,B)&= \int_\O |\Deriv \varphi-u^\perp B^T-v^\perp A^T|^2+2(\Deriv u-u^\perp A^T):v^\perp B^T+(\curl B)^2 \\
& \quad \quad +\frac{1}{\e^2}(1-|u|^2)|v|^2-\frac{2}{\e^2}(u\cdot v).
\end{align*}
The first and second outer variation of \(E_\e\) at \(u\) with respect to \(v\in \C^\infty_c(\O,\R^2)\), defined in an analogous manner as for \(\GL_\e\), are given by
\begin{align*}
\dd E_\e(u,v)=\int_\O \Deriv u :\Deriv v-\frac{1}{\e^2}(1-|u|^2)u\cdot v,
\end{align*}
\begin{equation}\nonumber
\dd^2 E_\e(u,v)= \int_\O |\Deriv v|^2+\frac{1}{\e^2}(1-|u|^2)|v|^2-\frac{2}{\e^2}(u\cdot v).
\end{equation}
\end{proposition}

Now we give a link between inner and outer variations when these quantities are computed at a smooth point, this link was previously observed in \cite{Le_2011,Le_2015,Le_Sternberg_2019}.

\begin{proposition}\label{prop:link_inner_outer}
Let \(\eta \in \C^\infty_c(\O,\R^2)\) and let \( (u,A)\in X \cap (\C^3(\O,\R^2))^2\) then 
\begin{align*}
\delta \GL_\e(u,A,\eta)= \dd \GL_\e \left(u,A,-\Deriv u .\eta,-\Deriv A .\eta+\Deriv \eta^T. A\right)
\end{align*}
\begin{align*}
\delta^2\GL_\e (u,A,\eta)&= \dd  \GL_\e \left(u,A,\Deriv^2u[\eta,\eta]+\Deriv u .\zeta, \ \Deriv^2A[\eta,\eta]+\Deriv A .\zeta+\Deriv \zeta ^T.A+2\Deriv \eta^T \Deriv A.\eta \right) \\
&\phantom{aaa} +\dd^2\GL_\e\left(u,A,-\Deriv u. \eta,-\Deriv A.\eta+\Deriv \eta^T.A\right).
\end{align*}

If \(u \in H^1(\O,\R^2)\cap \C^3(\O,\R^2)\) then
\begin{align*}
\delta E_\e(u,\eta)=\dd E_\e (u,-\Deriv u. \eta) 
\end{align*}
\begin{align*}
\delta^2E_\e(u,\eta)=\dd E_\e(u, \Deriv^2u[\eta,\eta]+\Deriv u.\zeta)+\dd^2E_\e(u, -\Deriv u. \eta).
\end{align*}
\end{proposition}

\begin{proof}
We first show that, for \(V\in \C^3(\O,\R^2)\) we have
\begin{align*}
V\circ \Phi_t^{-1} (y)=V(y)-t\Deriv V(y).\eta(y)+\frac{t^2}{2} X_0(y)+O(t^3)
\end{align*}
with \(X_0=\Deriv ^2V[\eta,\eta]+\Deriv V \Deriv \eta\). In order to do that we use the following Taylor expansion:
\begin{align*}
V\circ \Phi_t^{-1}(y)&=V(y)+t\p_t|_{t=0} (V\circ \Phi_t^{-1})(y) +\frac{t^2}{2}\p^2_{tt}|_{t=0}(V\circ \Phi_t^{-1})(y)+O(t^3)\\
&=V(y)+t\Deriv V(y).\p_t|_{t=0} \Phi_t^{-1}(y)+\frac{t^2}{2}\Bigl(\Deriv ^2V(y)[\p_t|_{t=0} \Phi_t^{-1}(y),\p_t|_{t=0} \Phi_t^{-1}(y)] \\
& +\Deriv V(y).\p^2_{tt}|_{t=0}\Phi_t^{-1}(y)\Bigr)+O(t^3).
\end{align*}
We first compute the derivatives with respect to \(t\) of \(\Phi_t^{-1}\). We use the expansion of \(\Phi_t\) given in \eqref{eq:expansion_Phi} and the relation
\begin{align*}
x=\Phi_t(\Phi_t^{-1}(x))=\Phi_t^{-1}(x)+t\eta (\Phi_t^{-1}(x))+\frac{t^2}{2}\Deriv \eta(\Phi_t^{-1}(x)).\eta(\Phi_t^{-1}(x))+O(t^3).
\end{align*}
Differentiating with respect to \(t\) yields
\begin{align*}
0=\p_t\Phi_t^{-1}(x)+t\Deriv \eta(\Phi_t^{-1}(x)).\p_t\Phi_t^{-1}+\eta(\Phi_t^{-1})+t\Deriv \eta(\Phi_t^{-1}(x)).\eta(\Phi_t^{-1}(x)) +O(t^2)
\end{align*}
and evaluating at \(t=0\) we find that \(\p_t|_{t=0}\Phi_t^{-1}(x)=-\eta(x)\). We can differentiate once more with respect to \(t\) to obtain
\begin{align*}
0=\p^2_{tt}\Phi_t^{-1}(x)+2\Deriv \eta(\Phi_t^{-1}(x)).\p_t \Phi_t^{-1}(x)+\Deriv \eta(\Phi_t^{-1}(x)).\eta(\Phi_t^{-1}(x))+O(t).
\end{align*}
Evaluating at \(t=0\) and using the expression previously found for \(\p_t|_{t=0}\Phi_t^{-1}(x)\) we arrive at \( \p^2_{tt}|_{t=0}\Phi_t^{-1}(x)=\Deriv \eta(x).\eta(x)\).
By the Taylor formula with integral remainder we know that 
\begin{equation}\label{eq:expansion_Phi_t-1}
\Phi_t^{-1}(x)=x-t\eta (x)+\frac{t^2}{2}\Deriv \eta (x).\eta (x)+O(t^3)
\end{equation} and we can check that we can differentiate under the integral sign giving the term in \(O(t^3)\) to obtain also that \(\Deriv \Phi_t^{-1}(x)=\Id-t\Deriv \eta (x)+\frac{t^2}{2}\Deriv \zeta(x)+O(t^3)\) where \(\zeta(x)=\Deriv \eta (x).\eta(x)\).
Thus we obtain
\begin{align}\label{eq:expansion_composition}
V\circ \Phi_t^{-1}=V-t\Deriv V.\eta+\frac{t^2}{2}\left(\Deriv^2V[\eta,\eta]+\Deriv V. (\Deriv \eta.\eta)\right)+O(t^3).
\end{align}

Now we recall from \eqref{eq:DL_pull_back} that, with some abuse of notation, \( (\Phi_t^{-1})^*A=\Deriv \Phi_t^{-T} (A\circ \Phi_t^{-1})\). 
 Thus by using the formula \eqref{eq:expansion_Phi_t-1}, we can write
\begin{align*}
(\Phi_t^{-1})^*A &=\left(\Id-t\Deriv \eta (x)+\frac{t^2}{2}\Deriv \zeta (x)+O(t^3) \right)^T \nonumber\\
& \quad \quad  \times\Bigl( A-t\Deriv A. \eta+\frac{t^2}{2}(\Deriv^2A[\eta,\eta]+\Deriv A. \zeta )+O(t^3)\Bigr)  \nonumber \\
&= A-t(\Deriv A.\zeta+(\Deriv \eta)^T.A) \\
&\quad \quad +\frac{t^2}{2}\Bigl( \Deriv^2A[\eta,\eta]+\Deriv A. \zeta +\Deriv \zeta^T.A+2\Deriv \eta^T .(\Deriv A .\eta) \Bigr) +O(t^3).
\end{align*}
Thus, if we let \( (u_t,A_t):= (u\circ \Phi_t^{-1},(\Phi_t^{1})^*A)\), by using \eqref{eq:expansion_composition} applied to \(V=u\) and by assuming that \( (u,A)\in (\C^\infty(\O,\R^2))^2\) we find that 
\begin{align*}
\GL_\e(u_t,A_t)&=\GL_\e\Bigl(u-t\Deriv u. \eta +\frac{t^2}{2}(\Deriv ^2u[\eta,\eta]+\Deriv u)+O(t^3),A-t(\Deriv A. \eta+\Deriv \eta^T.A)+ \\
& \quad \frac{t^2}{2}\left(\Deriv ^2A[\eta,\eta]+\Deriv A. \zeta+\Deriv \zeta^T.A+2\Deriv \eta^T.(\Deriv A.\eta)\right)+O(t^3) \Bigr) \\
&=\GL_\e(u,A)+t\dd \GL_\e(u,A,-\Deriv u.\eta,-\Deriv A.\eta+(\Deriv \eta)^TA) \\
& +\frac{t^2}{2}\dd \GL_\e \bigl(u,A,\Deriv ^2u[\eta,\eta]+\Deriv u. \zeta,\Deriv ^2A[\eta,\eta]+\Deriv A. \zeta+\Deriv \zeta^T.A+2\Deriv \eta^T.(\Deriv A.\eta)\bigr) \\
&\quad  +\frac{t^2}{2}\dd^2\GL_\e(u,A)(-\Deriv u.\eta,-\Deriv A.\eta+\Deriv \eta)^T.A) +O(t^3).
\end{align*}
By identification we conclude. A density argument allows us to extend this result for \( (u,A)\in (\C^3(\O,\R^2))^2\) Similar computations for \(E_\e\) give the result.
\end{proof}

Since critical points of the GL energy in the Coulomb gauge are smooth we can use Proposition \ref{prop:link_inner_outer} and we can deduce that stable critical points of \(\GL_\e\) satisfy that they have a non-negative second inner variation. This is summarized in the following corollary.

\begin{corollary}\label{cor:lin_inner_outer_GL}
Let \( (u,A)\) be in \(X\) such that \(\dd \GL_\e(u,A,v,B)=0\) for any\\ 
 \((v,B)\in (\C^\infty_c(\O,\R^2))^2\) and with \(A\) in the Coulomb gauge, then \(\delta \GL_\e(u,\eta)=0\) for any \(\eta\in \C^\infty_c(\O,\R^2)\). If we assume furthermore that \(\dd^2\GL_\e(u,A,v,B)\geq 0\) for any \( (v,B)\in (\C^\infty_c(\O,\R^2))^2\) then \( \delta^2 \GL_\e(u,A,\eta) \geq 0\) for any \(\eta \in \C^\infty_c(\O,\R^2)\). Similar results hold for the non-magnetic GL energy.
\end{corollary}


\subsection{Some remarks about inner variations }
The link between inner and outer variations for regular argument was already observed in \cite{Le_2011,Le_2015,Le_Sternberg_2019}. In order to make a direct link between the first and second inner variations when the argument is regular one can also use that for \(\eta \in \C^\infty_c(\O,\R^2)\)
\[ (\dive \eta )^2-\tr (\Deriv \eta)^2=2\det \Deriv \eta =\div [(\dive \eta)\eta-\Deriv \eta. \eta]\] and integrate by parts several times.

To examine the difference between inner and outer variations from the point of view of stability we can start by considering the 1D case. Let  \(\O=(a,b)\subset \R\) be an open interval with \(a<b\). By using e.g.\ \cite[Lemma 2.4]{Le_Sternberg_2019} we can show that for an energy of the form \( \E(V)=\int_a^b F(V,V')=\int_a^b \left(|V'|^2/2+f(V)\right) \dd x\) with \(f:\R\rightarrow \R\) a smooth function, the second inner variation is given by
\begin{equation}\nonumber
\delta^2 \E(V,\eta)=\int_\O \p^2_{pp}F(V,V')[V',V']|\eta'|^2= \frac12 \int_a^b |\eta'|^2|V'|^2.
\end{equation}
 for all \(\eta \in \C^\infty_c((a,b),\R)\). Surprisingly, this quantity does not depend on \(f\) and is always non-negative. This allows us to recover the following known result about strictly monotone solutions of EDO in 1D.

\begin{proposition}
Let \(V\in \C^2((a,b),\R)\) be a critical point of \(\E(V)=\int_a^b \left(|V'|^2/2+f(V)\right) \) with \(f\in \C^\infty(\R,\R)\), i.e.\ a solution of \(-V''+f'(V)=0\) in \((a,b)\). Assume furthermore that \(V\) is strictly monotone, then \(V\) is stable, i.e.\
\begin{equation}\nonumber
\int_a^b \left(|\varphi'|^2+f''(V)\varphi^2 \right) \geq 0, \quad \forall \varphi \in \C^\infty_c((a,b)).
\end{equation}
\end{proposition}

\begin{proof}
We first observe that \(V\) is in \(\C^\infty( (a,b)\). Then every \(\varphi\in \C^\infty_c((a,b))\) can be written as \( \varphi=V'\eta\) since \(V'\) does not vanish in \((a,b)\). We can thus use a result analogous to Proposition \ref{prop:link_inner_outer}, see e.g.\ \cite[Corollary 2.3]{Le_Sternberg_2019},   and the fact that \( \dd E(V, \Deriv^2V[\eta,\eta]+\Deriv V. (\Deriv \eta. \eta))=0\) since \(\Deriv^2V[\eta,\eta]+\Deriv V. (\Deriv \eta. \eta)  \in \C^\infty_c( (a,b),\R)\) to conclude that 
\[ \dd^2 E (V,\varphi)=\delta^2(E,-\frac{\varphi}{V'}) =\int_a^b\p^2_{pp}F(V,V')[V',V']|\left(\frac{\varphi}{V'}\right)'|^2 \geq 0\]
for all \(\varphi \in \C^\infty_c((a,b))\).
\end{proof}

For a classical proof of the above fact we refer to Proposition 1.2.1 and Definition 1.2.1 in \cite{Dupaigne_2011}.

\section{Passing to the limit in the second inner variation}\label{sec:pass_limit}

From the expression of the second inner variation of \(\GL_\e\) given in Proposition \ref{prop:inner_variations_GL_magnetic} it appears that to understand the limit of \(\delta^2\GL_\e(u_\e,A_\e,\eta)/h_{\ex}^2\) for \( \{(u_\e,A_\e)\}_{\e>0}\) a family of critical points of the GL energy we need to understand the limit of all the quadratic terms in the derivatives \(|\p_1^{A_\e}u_\e|^2/h_{\ex}^2, |\p_2^{A_\e}u_\e|^2/h_{\ex}^2\) and \( \langle \p_1^{A_\e}u_\e, \p_2^{A_\e} u_\e\rangle /h_{\ex}^2\). This is the object of this section.

\subsection{The case with magnetic field}

The following proposition is mainly the lower-bound for the \(\Gamma\)-convergence result of \(\GL_\e/h_{\ex}^2\) obtained in \cite{Sandier_Serfaty_2007}, we present the proof here to underline the fact that thanks to assumption \eqref{cond:convergence_energies}  we know the limit of the energy density.

\begin{proposition}\label{prop:lower_bound_with}
Let \( \{ (u_\e,A_\e)\}_{\e>0}\) be a family of critical points of \(\GL_\e\) satisfying \eqref{eq:energy_bound}-\eqref{eq:magnetic_field_bound}. We set \(j_\e=\langle iu_\e,(\nabla -iA_\e)u_\e\rangle\) and \(h_\e=\curl A_\e\).
\begin{itemize}
\item[1)] Up to a subsequence,
\begin{equation}\nonumber
\frac{\mu(u_\e,A_\e)}{h_{\ex}}:=\frac{\curl (j_\e+A_\e)}{h_{\ex}} \xrightarrow[\e \to 0]{} \mu \quad \text{ in } \left( \C^{0,\gamma}(\O)\right)^*
\end{equation}
for every \(\gamma \in (0,1)\) and 
\begin{equation}\nonumber
\frac{j_\e}{h_{\ex}}\xrightharpoonup[\e \to 0]{} j, \quad \frac{h_\e}{h_{\ex}}\xrightharpoonup[\e \to 0]{} h \quad \text{ in } L^2(\O)
\end{equation}
with \(-\nabla^\perp h=j\) and \(\mu=-\Delta h+h\). Furthermore
\begin{align*}\nonumber
\liminf_{\e \to 0} \frac{\GL_\e(u_\e,A_\e)}{h_{\ex}^2} &\geq \liminf_{\e \to 0} \frac{1}{2h_{\ex}^2}\int_\O \left(  |\nabla h_\e|^2+|h_\e-h_{\ex}|^2 \right) \\
& \geq \frac{|\mu|(\O)}{2\lambda}+\frac12 \int_\O \left( |\nabla h|^2+|h-1|^2\right).
\end{align*}

\item[2)] We set \( g_\e(u_\e,A_\e) := \frac12 \left( |\nabla u_\e|^2+|h_\e-h_{\ex}|^2+\frac{1}{2\e^2}(1-|u_\e|^2)^2 \right)\). Let us assume that \eqref{cond:convergence_energies} holds then, 
\begin{equation}\label{eq:weak_conv_meas}
\frac{g_\e(u_\e,A_\e)}{h_{\ex}^2} \rightharpoonup \frac{1}{2\lambda}|\mu|+\frac{1}{2}\left( |\nabla h|^2+|h-1|^2\right) \text{ in } \M(\O),
\end{equation}
\begin{align}
\frac{|\nabla h_\e|^2}{|u_\e|^2h_{\ex}^2} \rightharpoonup |\nabla h|^2+\frac{1}{\lambda}|\mu|, \quad  
\frac{|\nabla h_\e|^2}{h_{\ex}^2} \rightharpoonup |\nabla h|^2+\frac{1}{\lambda}|\mu|, \quad
\frac{|\nabla_{A_\e}u_\e|^2}{h_{\ex}^2} \rightharpoonup |\nabla h|^2+\frac{1}{\lambda}|\mu| \label{eq:conv_meas_h_hrho_} 
\end{align}
and
\begin{equation}\label{eq:important}
\frac{1}{h_{\ex}^2}\left(|\nabla |u_\e||^2+\frac{1}{2\e^2}(1-|u|^2)^2 \right) \rightharpoonup 0\ \text{ in } \M(\O).
\end{equation}
\end{itemize}
\end{proposition}

\begin{proof}
We recall that if \( (u,A)\) is a solution to \eqref{eq:GL_equations_magnetic} then \(|u|\leq 1\) in \(\O\), see e.g.\ \cite[Chapter 3]{Sandier_Serfaty_2007}. We also observe that near points where \(u\) does not vanish we can write \(u=\rho e^{i \varphi}\). Even if the phase \(\varphi\) is not globally defined it can be seen that its gradient is globally defined. Using the second equation in \eqref{eq:GL_equations_magnetic} we find that 
\[ -\nabla^\perp h=\rho^2(\nabla \varphi-A).\]
We can also see that
\[ |\nabla_Au|^2=|\nabla |u||^2+\rho^2|\nabla \varphi-A|^2,\] 
\begin{align}
\GL_\e(u_\e,A_\e) &= \frac12 \int_\O |\nabla |u_\e||^2+|u_\e|^2|\nabla \varphi_\e-A_\e|^2+|h_\e-h_{\ex}|^2+\frac{1}{2\e^2}(1-|u_\e|^2)^2 \nonumber \\
&=\frac12 \int_\O |\nabla |u_\e||^2+\frac{|\nabla h_\e|^2}{|u_\e|^2}+|h_\e-h_{\ex}|^2+\frac{1}{2\e^2}(1-|u_\e|^2)^2 \label{eq:1st_equality_imp} \\
& \geq \frac12 \int_\O |\nabla |u_\e||^2+|\nabla h_\e|^2+|h_\e-h_{\ex}|^2+\frac{1}{2\e^2}(1-|u_\e|^2)^2. \label{eq:2nd_ineq_imp}
\end{align}
Then, we can use the energy bound \(\GL_\e(u_\e,A_\e)\leq C h_{\ex}^2\)  to deduce that \( h_\e/h_{\ex}\) is bounded in \(H^1(\O)\) and thus, converges weakly in that space, up to a subsequence, to some \(h\in H^1(\O)\). We also observe that, since we consider solutions to \eqref{eq:GL_equations_magnetic}, then \(j_\e=-\nabla^\perp h_\e\)  and \(\mu(u_\e,A_\e)=\curl j_\e+h_\e=-\Delta h_\e+h_\e \rightharpoonup-\Delta h+h=\mu\) in \(H^{-1}(\O)\). We now show the convergence of \(\mu_\e:=\mu(u_\e,A_\e)\) in \((\C^{0,\gamma}(\O))^*\) and the lower bound. Since we assume  in \eqref{eq:magnetic_field_bound} that \(h_{\ex}\leq C |\log \e|\) we have from \eqref{eq:energy_bound} that \(\GL_\e(u_\e,A_\e)\leq Ch_{\ex}^2\leq C|\log \e|^2\). We can then apply Proposition 1.1 in \cite{Sandier_Serfaty_2000b} (see also \cite[Theorem 4.1]{Sandier_Serfaty_2007})\footnote{The reason why we refer to \cite{Sandier_Serfaty_2000b} is that the lower bound is explicitly stated in terms of \(\int_{\cup B_i} |\nabla h_\e|^2\) there and not in terms of the full energy.} to find a family of balls (depending on \(\e\)) \( (B_i)_{i\in I_\e}=(B(a_i,r_i))_{i\in I_\e}\) such that 
\begin{equation*}
\{x ; |u_\e(x)|\leq \frac{1}{2}\} \subset \bigcup_{i\in I_\e} B(a_i,r_i),
\end{equation*}
\begin{equation*}
\sum_{i\in I_\e} r_i \leq \frac{1}{|\log \e|^6}
\end{equation*}
\begin{equation*}
\frac12 \int_{B_i} |\nabla h_\e|^2 \geq \pi |d_i| |\log \e|(1-o_\e(1))
\end{equation*}
with \(h_\e=\curl A_\e\), \( d_i=\deg (\frac{u_\e}{|u_\e|},\p B_i)\) if \(\overline{B}_i \subset
 \O\) and \(0\) otherwise.
 
 We let \(V_\e := \bigcup_{i\in I_\e} B_i\), then by using \eqref{eq:2nd_ineq_imp} we find
 \begin{align*}
 \GL_\e(u_\e,A_\e,V_\e) \geq \frac12 \int_{V_\e} |\nabla h_\e|^2  \geq  \pi \sum_{i\in I_\e} |d_i||\log \e|\left(1- o_\e(1)\right).
 \end{align*}
Note that \eqref{eq:2nd_ineq_imp} and \eqref{eq:energy_bound} imply that \(\sum_{i\in I_\e} |d_i|\leq C|\log \e|\leq C h_{\ex}\). Now let \(U\) be an open sub-domain of \(\O\), working in \(U\) will be useful to prove point 2). 
We can write
\begin{align}
\GL_\e(u_\e,A_\e,U)&=\GL_\e(u_\e,A_\e,V_\e)+\GL_\e(u_\e,A_\e,U\setminus V_\e) \nonumber \\
& =\frac12 \int_{V_\e} |\nabla h_\e|^2+ \frac12 \int_{U\setminus V_\e} \left(|\nabla h_\e|^2 +|h_\e-h_{\ex}|^2 \right)  \label{eq:intermediate} \\
&\geq \pi \sum_i |d_i| |\log \e|+\frac12 \int_{U\setminus V_\e} \left(|\nabla h_\e|^2 +|h_\e-h_{\ex}|^2 \right)-o(h_{\ex}^2) \label{eq:pre_delta}.
\end{align}
We divide by \(h_{\ex}^2\) to obtain 
\begin{align}
\frac{\GL_\e(u_\e,A_\e,U)}{h_{\ex}^2} & \geq \frac{1}{2h_{\ex}^2} \int_U |\nabla h_\e|^2+|h_\e-h_{\ex}|^2  \label{eq:delta}\\
&\geq \pi \frac{\sum_i |d_i|}{h_{\ex}} \frac{|\log \e|}{h_{\ex}} +\int_{U\setminus V_\e} \left| \frac{\nabla h_\e}{h_{\ex}} \right|^2+\left|\frac{h_\e}{h_{\ex}}-1 \right|^2-o(1). \nonumber
\end{align}
Since \(\sum_{i\in I_\e} r_i\xrightarrow[\e \to 0]{} 0\) we can extract a subsequence \(\e_n \to 0\) such that, if we set \(\mathcal{A}_N:=\bigcup_{n \geq N} V_{\e_n}\) we have \(|\mathcal{A}_N|\rightarrow 0\) when \(N \to +\infty\). By weak convergence of \(h_\e\) in \(H^1(\O)\), for every \(N\) fixed
\begin{align*}
\liminf_{n \to +\infty} \int_{U\setminus V_{\e_n}} \left| \frac{\nabla h_{\e_n}}{h_{\ex}}\right|^2+\left|\frac{h_{\e_n}}{h_{\ex}}-1 \right|^2 &\geq \liminf_{n \to +\infty} \int_{U\setminus \mathcal{A}_N} \left| \frac{\nabla h_{\e_n}}{h_{\ex}}\right|^2+\left|\frac{h_{\e_n}}{h_{\ex}}-1 \right|^2 \\
&\geq \int_{U\setminus \mathcal{A}_N} |\nabla h|^2+|h-1|^2.
\end{align*}
We then pass to the limit \(N \to +\infty\) to find
\begin{equation}\label{eq:carre}
\liminf_{n \to +\infty} \int_{U\setminus V_{\e_n}} \left| \frac{\nabla h_{\e_n}}{h_{\ex}}\right|^2+\left|\frac{h_{\e_n}}{h_{\ex}}-1 \right|^2 \geq \int_U |\nabla h|^2+|h-1|^2.
\end{equation}
On the other hand, coming back to \eqref{eq:pre_delta} and using that \(\GL_\e(u_\e,A_\e)\leq Ch_{\ex}^2\) we find that \(\frac{1}{h_{\ex}}\sum_{i\in I}|d_i|\) stays bounded. Hence \(\frac{2\pi}{h_{\ex}} \sum_{i} d_i \delta_{a_i}\) converges, up to a subsequence in \( (\C^0_0(U))^*)\). We then use the Jacobian estimate of Theorem 6.1 in \cite{Sandier_Serfaty_2007} in \(U\) to say that this limit is also the limit of \(\mu_\e\) and thus is equal to \(\mu=-\Delta h+h\). Theorem 6.2 in \cite{Sandier_Serfaty_2007} applied in \(\O\) implies that \(\mu_\e\) converges towards \(\mu\) in \((\C^{0,\gamma}_0(\O))^*\). We then pass to the limit in \eqref{eq:delta} and we use \eqref{eq:carre} to obtain
\begin{align}
\liminf_{n \to +\infty} \frac{\GL_{\e_n}(u_{\e_n},A_{\e_n},U)}{h_{\ex}^2} &\geq  \frac{1}{2h_{\ex}^2}\int_U \left( |\nabla h_{\e_n}|^2+|h_{\e_n}-h_{\ex}|^2 \right)  \nonumber \\
& \geq \frac{1}{2\lambda}|\mu|(U)+\frac12 \int_U \left(|\nabla h|^2+|h-1|^2 \right).\label{eq:***}
\end{align}
This proves point 1).

\medskip
 To prove point 2) we assume that \eqref{cond:convergence_energies} holds. Then we set \[ g_\e := \frac12 \left( |\nabla u_\e|^2+|h_\e-h_{\ex}|^2+\frac{1}{2\e^2}(1-|u_\e|^2)^2 \right) \dd x\] we have that \(g_\e (\O)\rightarrow \left(\frac{1}{2\lambda}|\mu|+\frac{1}{2}\left( |\nabla h|^2+|h-1|^2\right) \right)(\O)\) and \[ \liminf_{\e \to 0} g_\e (U) \geq \left(\frac{1}{2\lambda}|\mu|+\frac{1}{2}\left( |\nabla h|^2+|h-1|^2\right) \right)(U)\] for every open set \(U\subset \O\). We can then apply Proposition 1.80 in  \cite{Ambrosio_Fusco_Pallara_2000} to deduce that \eqref{eq:weak_conv_meas} holds. By using \eqref{eq:1st_equality_imp}-\eqref{eq:2nd_ineq_imp} and the strong convergence of \(\frac{h_\e}{h_{\ex}}\) in \(L^2(\O)\) we also arrive at \eqref{eq:conv_meas_h_hrho_} and \eqref{eq:important}.
\end{proof}

We are now ready to examine the convergence of the quadratic terms appearing in the formula for the second inner variation in Proposition \ref{prop:inner_variations_GL_magnetic}.
\begin{proposition}\label{prop:conv-quadratic_term_conv_magnetic}
Let \( \{(u_\e,A_\e)\}_{\e>0}\) be a family of critical points of \(\GL_\e\)  satisfying \eqref{eq:energy_bound} and \eqref{eq:magnetic_field_bound}. Let us assume that \eqref{cond:convergence_energies} holds, then, either the limiting vorticity is constant equal to \(1\) in all of \(\O\) or, in the sense of measures,
\begin{equation}\nonumber
\frac{|\p_1^{A_\e}u_\e|^2}{h_{\ex}^2} \rightharpoonup |\p_2h|^2+\frac{|\mu|}{2\lambda}, \quad \frac{|\p_2^{A_\e}u_\e|^2}{h_{\ex}^2} \rightharpoonup |\p_1h|^2+\frac{|\mu|}{2\lambda}
\end{equation}
\begin{equation}\nonumber
\frac{\langle \p_1^{A_\e}u_\e,\p_2^{A_\e} u_\e\rangle }{h_{\ex}^2} \rightharpoonup -\p_1h\p_2h.
\end{equation}
\end{proposition}

\begin{proof}
Thanks to \eqref{eq:energy_bound} the measures \(\frac{|\p_1^{A_\e}u_\e|^2}{h_{\ex}^2}, \frac{|\p_2^{A_\e}u_\e|^2}{h_{\ex}^2} \) and \(\frac{\langle \p_1^{A_\e}u_\e,\p_2^{A_\e}u_\e \rangle}{h_{\ex}^2} \) are bounded. Thus there exist \(\nu_1,\nu_2,\nu_3\) in \(\M(\O)\) such that, in the sense of measures,
\begin{align*}
\frac{|\p_1^{A_\e}u_\e|^2}{h_{\ex}^2} \rightharpoonup |\p_2h|^2+\nu_1, \quad \frac{|\p_2^{A_\e}u_\e|^2}{h_{\ex}^2} \rightharpoonup |\p_1h|^2+\nu_2, \quad 
\frac{\langle \p_1^{A_\e}u_\e,\p_2^{A_\e} u_\e\rangle}{h_{\ex}^2} \rightharpoonup -\p_1h\p_2h+\nu_3.
\end{align*}

We use that from Corollary \ref{cor:lin_inner_outer_GL} we have that \(\delta \GL_\e(u_\e,A_\e,\eta)=0\) for all \(\eta\in \C^\infty_c(\O,\R^2)\) and this implies, thanks to the expressions in Proposition \ref{prop:inner_variations_GL_magnetic}, that \(\frac{1}{h_{\ex}^2}\dive (T_\e)=0\) in \(\O\), where \(T_\e\) is defined in \eqref{eq:tensor_GL_magnetic}. Then, by using \eqref{eq:important}, we pass to the limit when \(\e \to 0\) in the sense of distributions to find that 
\begin{equation}
-\dive(T_h)+\dive \begin{pmatrix}
\nu_1-\nu_2 & \nu_3 \\
\nu_3 & \nu_2-\nu_1
\end{pmatrix}=0,
\end{equation}
where \((T_h)_{ij}= \p_ih \p_jh -\frac12 (|\nabla h|^2+h^2)\delta_{ij}\). But we can use Theorem \ref{th:critical_points} obtained in \cite{Sandier_Serfaty_2003,Sandier_Serfaty_2007} to say that \(\dive (T_h)=0\) and  deduce that \(\dive \begin{pmatrix}
\nu_1-\nu_2 & \nu_3 \\
\nu_3 & \nu_2-\nu_1
\end{pmatrix}=0\). This equation can be rewritten as the Cauchy-Riemann  system
\begin{equation}\nonumber
\left\{
\begin{array}{rcll}
\p_1(\nu_1-\nu_2)+\p_2\nu_3 &=&0 \\
\p_1 \nu_3 -\p_2 (\nu_1-\nu_2) &=&0
\end{array}
\right.
\end{equation}
or \(\p_{\bar{z}} (\nu_1-\nu_2-i\nu_3)=0\) where \(\p_{\bar{z}}= \frac12 (\p_1+i\p_2)\). Since the operator \(\p_{\bar{z}}\) is elliptic we deduce that \(\nu_1-\nu_2-i\nu_3\) is holomorphic in \(\O\). Now we can show that if \(\overline{\O}=\supp \mu\) then \(h\) is constantly equal to \(1\). Indeed, by contradiction if there exists \(x_0\in \O\) such that \(|\nabla h(x_0)|\neq 0\) then from \cite[Theorem 3.1]{Rodiac_2019}\footnote{This result is recalled in the appendix for the comfort of the reader} there exists a neighbourhood \(\omega_{x_0}\) of \(x_0\) in which we have \(\mu=\pm 2|\nabla h| \mathcal{H}^1_{\lfloor \supp \mu \cap\{ |\nabla h|>0\}}\) with \( \supp \mu \cap\{ |\nabla h|>0\}\) which is a \(\C^1\) curve. Hence we find that \(|\mu|\) vanishes in a small ball included in \(\omega_{x_0}\) and not intersecting this curve. This is a contradiction and thus we find that \(h\) is constant, and \(h\) being equal to \(1\) on \(\p \O\), we conclude that \(h=1\) and \(\mu=-\Delta h+h=1\) in \(\O\). 

Hence if \(h \neq 1\) then \(\supp \mu \neq \overline{\O}\) and thus there exists a ball \(B\subset \O\) such that \(|\mu|_{\lfloor B}=0\). We thus deduce from \eqref{eq:conv_meas_h_hrho_} that \(h_\e/h_{\ex}\) converges strongly to \(h\) in \(B\) and \(\nu_1=\nu_2=0\) in \(B\) since \(\nu_1+\nu_2=|\mu|/\lambda\) and \(\nu_1,\nu_2\geq 0\).

From \eqref{eq:conv_meas_h_hrho_} we also find that \(\frac{|\nabla h_\e|^2}{h_{\ex}^2|u_\e|^2}\rightharpoonup |\nabla h|^2\) in \(B\). Since \(|\nabla h|^2 \dd x\) does not charge the boundary \(\p B\) from \cite[Theorem 1.40]{Evans_Gariepy_2015} we deduce that 
\[ \int_B \frac{|\nabla h_\e|^2}{h_{\ex}^2|u_\e|^2} \rightarrow \int_B |\nabla h|^2.\]
 Since \(h_\e/h_{\ex} \rightarrow h \) in \(H^1(B)\) we can also assume that, up to a subsequence, \(\nabla h_\e/h_{\ex} \rightarrow \nabla h\) a.e.\ in \(B\). From the energy bound \eqref{eq:energy_bound} we also know that \( |u_\e|^2\rightarrow 1\) in \(L^2(\O)\) and hence, up to a subsequence, \(|u_\e|\rightarrow 1\) a.e.\ Hence Brezis-Lieb's lemma implies that 
\begin{equation}\label{eq:strong_conv_h/rho}
 \frac{\nabla h_\e}{h_{\ex}|u_\e|} \rightarrow \nabla h \quad \text{ in } L^2(B).
 \end{equation}
 
 Now if we write, locally near a point where \(u_\e\) does not vanish, \(u_\e=\rho_\e e^{i\varphi_\e}\) then
\begin{align*}
\p_j^{A_\e} u_\e = \p_j u_\e-iA_\e^j u_\e &=\p_j \rho_\e e^{i\varphi_\e}+iu_\e (\p_j \varphi_\e-iA_\e)
\end{align*}
and
\begin{align*}
\langle \p_1^{A_\e} u_\e, \p_2^{A_\e} u_\e \rangle =\p_1 \rho_\e \p_2 \rho_\e +\rho_\e^2 (\p_1 \varphi_\e-A_1^\e)(\p_2 \varphi_\e-A_2^\e).
\end{align*}
Recalling that \(-\nabla^\perp h_\e= \rho_\e^2(\nabla \varphi_\e-A_\e)\) we arrive at 
\begin{align*}
\langle \p_1^{A_\e} u_\e,\p_2^{A_\e} u_\e\rangle =\p_1 |u_\e| \p_2 |u_\e|-\frac{\p_2 h_\e \p_1 h_\e}{|u_\e|^2}.
\end{align*}
We use \eqref{eq:important} to infer that \( \rho_\e/h_{\ex} \rightarrow 0\) strongly in \(H^1(\O)\) and then we use this together with \eqref{eq:strong_conv_h/rho} to find that \(\frac{1}{h_{\ex}^2}\langle \p_1^{A_\e} u_\e, \p_2^{A_\e} u_\e \rangle \rightarrow -\p_2h \p_1h\) in \(L^1(B)\). This implies that \(\nu_3=0\) in \(B\).

We have thus obtained that \(\nu_1-\nu_2-i\nu_3\) vanishes in the ball \(B\). This quantity being holomorphic, the principle of isolated zeros implies that \(\nu_1=\nu_2\) and \(\nu_3=0\) everywhere in \(\O\). Since \(\nu_1+\nu_2= |\mu|/\lambda\) we find that \(\nu_1=\nu_2=\frac{|\mu|}{2\lambda}\) and \(\nu_3=0\).
\end{proof}

\subsection{The case without magnetic field}

In this section we state the analogue of Proposition \ref{prop:lower_bound_with} and \ref{prop:conv-quadratic_term_conv_magnetic} in the case of the GL energy without magnetic field. Since the proofs require only minors adaptations of the previous paragraph they are left to the reader.

\begin{proposition}
Let \( \{u_\e\}_{\e>0}\) be a family of critical points of \(E_\e\) satisfying \eqref{eq:energy_bound_without}. We set \(\tilde{j}_\e=\langle iu_\e,\nabla u_\e\rangle \) and \(U_\e\in H^1(\O)\) the unique function such that \(\nabla^\perp U_\e=\tilde{j}_\e\) and \(\int_\O U_\e =0\).
\begin{itemize}
\item[1)] Up to a subsequence,
\begin{equation}
\frac{\mu(u_\e)}{|\log \e|}:=\frac{\curl \tilde{j}_\e}{|\log \e|} \xrightarrow[\e \to 0]{} \mu \quad \text{ in } \left( \C^{0,\gamma}(\O)\right)^*
\end{equation}
for every \(\gamma \in (0,1)\) and 
\begin{equation}
\frac{\tilde{j}_\e}{|\log \e|}\xrightharpoonup[\e \to 0]{} j, \quad \frac{U_\e}{|\log \e|}\xrightharpoonup[\e \to 0]{} h \quad \text{ in } L^2(\O)
\end{equation}
with \(-\nabla^\perp U=\tilde{j}\) and \(\mu=-\Delta U\). Furthermore
\begin{equation}
\liminf_{\e \to 0} \frac{E_\e(u_\e)}{|\log \e|^2}\geq \frac{|\mu|(\O)}{2}+\frac12 \int_\O |\nabla U|^2.
\end{equation}

\item[2)] Let us assume that \eqref{cond:conv_energy_2} holds then, if we set \(e_\e(u):=\frac12 \left(|\nabla u|^2+\frac{1}{2\e^2}(1-|u|^2)^2 \right)\) then
\begin{equation}\label{eq:weak_conv_meas_2}
\frac{e_\e(u_\e)}{|\log \e|^2} \rightharpoonup \frac{1}{2}|\mu|+\frac{1}{2}\left( |\nabla U|^2\right) \text{ in } \M(\O),
\end{equation}
\(\frac{|\nabla U_\e|^2}{|\log \e|^2}\rightharpoonup |\nabla U|^2+|\mu| \text{ in } \M(\O)\) and \(\frac{|\nabla U_\e|^2}{|u_\e|^2|\log \e|^2}\rightharpoonup |\nabla U|^2+|\mu| \text{ in } \M(\O)\).
\end{itemize}
\end{proposition}

\begin{proposition}\label{prop:conv_quadratic_without}
Let \( \{u_\e\}_{\e>0}\) be a family of critical points of \(E_\e\). Let us assume that \eqref{cond:conv_energy_2} holds, then, in the sense of measures,
\begin{equation}
\frac{|\p_1u_\e|^2}{|\log \e|^2} \rightharpoonup |\p_2U|^2+\frac{|\mu|}{2}, \quad  \frac{|\p_2u_\e|^2}{|\log \e|^2} \rightharpoonup |\p_1U|^2+\frac{|\mu|}{2}
\end{equation}
\begin{equation}
\frac{\langle \p_1u_\e,\p_2 u_\e \rangle}{|\log \e|^2} \rightharpoonup -\p_1U \p_2U.
\end{equation}
\end{proposition}

\begin{proof}
The proof follows the same lines as the proof of  Proposition \ref{prop:conv-quadratic_term_conv_magnetic}. However we use \cite[Theorem 1.3]{Rodiac_2016}\footnote{cf.\ Appendix.} to say that \(\mu\) is locally supported on a union of curves instead of the results in \cite[Theorem 3.1]{Rodiac_2019}. 
\end{proof}

\subsection{Proofs of the mains theorems}

We are now ready to prove Theorem \ref{th:main1}.

\begin{proof}(proof of Theorem \ref{th:main1})
Let \(\eta \in \C^\infty_c(\O,\R^2)\). We want to understand the limit of 
\begin{align*}
\frac{\delta^2 \GL_\e(u_\e,A_\e,\eta)}{h_{\ex}^2}&=\frac{1}{h_{\ex}^2}\int_\O \Bigl(|(\Deriv u-u^\perp A^T)\Deriv \eta|^2-|\Deriv u-u^\perp A^T|^2 \det \Deriv \eta \\
& \quad +h^2 \left[ (\dive \eta)^2-\det \Deriv \eta \right] +\frac{1}{\e^2}(1-|u|^2)^2 \det \Deriv \eta \Bigr).
\end{align*}
We note that 
\begin{align*}
|(\Deriv u-u^\perp A^T)\Deriv \eta|^2= |\p_1^Au|^2 |\nabla \eta_1|^2+|\p_2^Au|^2 |\nabla \eta_2|^2+2\langle \p_1^A u,\p_2^Au\rangle \nabla \eta_1\cdot \nabla \eta_2.
\end{align*}
Now, since \(|\nabla_A u|^2/h_{\ex}^2\) and \((1-|u|^2)^2/h_{\ex}^2\) are bounded sequences in \(L^1(\O)\) we can extract subsequences for which we have the following convergence in the sense of measures:
\begin{align*}
\frac{|\p_1^{A_\e} u_\e|^2}{h_{\ex}^2} \rightharpoonup |\p_2h|^2 +\nu_1, \quad 
\frac{|\p_2^{A_\e} u_\e|^2}{h_{\ex}^2}\rightharpoonup |\p_1h|^2 +\nu_2, \\
\frac{\langle \p_1^{A_\e}u_\e,\p_2^{A_\e}u_\e\rangle }{h_{\ex}^2}\rightharpoonup -\p_1h \p_2h+\nu_3, \quad 
\frac{(1-|u|^2)^2}{\e^2 h_{\ex}^2} \rightharpoonup \nu_4,
\end{align*}
with \(\nu_1,\nu_2,\nu_3,\nu_4 \in \mathcal{M}(\O)\). By using that \(h_\e/h_{\ex}\) is bounded in \(H^1(\O)\) we also have that, up to a subsequence \(h_\e \rightarrow h\) strongly in \(L^2(\O)\). Thus we can pass to the limit and we find that 
\begin{align*}
\frac{\delta^2 \GL_\e(u_\e,A_\e,\eta)}{h_{\ex}^2} &\xrightarrow[\e \to 0]{} \int_\O \Bigl[ (|\p_2 h|^2+\nu_1) |\nabla \eta_1|^2+(|\p_1h|^2|+\nu_2) |\nabla \eta_2|^2\\
& -(2\p_1h\p_2h-2\nu_3) \nabla \eta_1\cdot \nabla \eta_2  \\
&+(|\nabla h|^2+\nu_1+\nu_2) \det \Deriv \eta+h^2\left[(\dive \eta)^2-\det \Deriv \eta \right]+ \nu_4 \det \Deriv \eta \Bigr].
\end{align*}
Now, if we assume the convergence of energy \eqref{cond:convergence_energies} then \eqref{eq:important} and  Proposition \ref{prop:conv-quadratic_term_conv_magnetic} give that \(\nu_1=\nu_2=|\mu|/2\lambda\), \(\nu_3=0\) and \(\nu_4=0\). This allows us to rewrite
\begin{align*}
\lim_{\e \to 0} \frac{\delta^2 \GL_\e(u_\e,A_\e,\eta)}{h_{\ex}^2} &= \int_\O \Bigl( |\p_2 h|^2 |\nabla \eta_1|^2+|\p_1h|^2| \nabla \eta_2|^2-2\p_1h\p_2h \nabla \eta_1\cdot \nabla \eta_2  \\
& +|\nabla h|^2 \det \Deriv \eta +h^2\left[(\dive \eta)^2-\det \Deriv \eta \right]\Bigr)+\int_\O \left( \frac{|\Deriv \eta|^2}{2}- \det \Deriv \eta \right) \frac{\dd|\mu|}{\lambda}.
\end{align*}
We can conclude since \( |\Deriv \eta^T \nabla^\perp h|^2= |\p_2 h|^2 |\nabla \eta_1|^2+|\p_1h|^2| \nabla \eta_2|^2-2\p_1h\p_2h \nabla \eta_1\cdot \nabla \eta_2\).
To finish the proof we need to show the validity of \eqref{eq:stability_condition_magnetic}, this is a consequence of the link between inner and outer variations cf.\ Corollary \ref{cor:lin_inner_outer_GL}, the definition of stability and the limit of the second inner variation previously obtained.
\end{proof}

The proof of Theorem \ref{th:main_2_without} follows the same lines by using Proposition \ref{prop:conv_quadratic_without} and is left to the reader.

\section{Analysing the limiting stability condition}\label{sec:Analysis_limiting_condition}

As a consequence of Corollary \ref{cor:lin_inner_outer_GL} and Theorem \ref{th:main1} we can see that if \(\{(u_\e,A_\e)\}_{\e>0}\) is a family of stable critical points of \(\GL_\e\) then \(Q_h(\eta)\geq 0\) for every \(\eta \in \C^\infty_c(\O,\R^2)\), with \(Q_h\) defined in \eqref{eq:limiting_stability}. We would like to analyse if this limiting stability condition implies more regularity on the limiting vorticity. In the case with magnetic field we take a specific example of an admissible limiting vorticity supported on a line in the Lipschitz bounded domain\footnote{Even if we assumed \(\O\) smooth at the beginning it can be seen that our analysis is still valid for such Lipschitz domains.} \(\O=(-L,L)^2\) and we show that the associated limiting magnetic field \(h\) satisfies that \(Q_h(\eta)\geq 0\) for every \(\eta\in \C^\infty_c(\O,\R^2)\) if \(L>0\) is small enough whereas there exists \(\eta\in \C^\infty_c(\O,\R^2)\) such that \(Q_h(\eta)<0\) for \(L\) large enough. This shows that the link between limiting stability of the vorticity measure and regularity might be subtle and may depend on other factors such as the size of the domain. In the case without magnetic field the situation is even worse in a sense. Indeed we can use a result of Iwaniec-Onninen \cite{Iwaniec_Onninen_2022} to prove that every limiting vorticity measure satisfies that \(\tilde{Q}_U(\eta)\geq 0\) for every \(\eta\in \C^\infty_c(\O,\R^2)\) where \(\tilde{Q}_U\) is defined in \eqref{eq:stability_limit_without}. This shows that no supplementary regularity can be obtained from our limiting stability condition in that case.

\subsection{The case with magnetic field}

\begin{proposition}\label{prop:analyze_with_magnetic}
Let \(\O=(-L,L)^2\), we set \(h(x,y)=e^{-|x|}\) for \( (x,y)\in \O\). Then \(h\) satisfies \(-\Delta h+h=\mu\) in \(\O\) with \(\mu=-2\mathcal{H}^1_{\lfloor \{x=0\}}\) and \(h\) satisfies \eqref{eq:stress_energy}. With \(Q_h\) defined in \eqref{eq:limiting_stability} we have that
\begin{itemize}
\item[1)] if \(L>0\) is small enough then \(Q_h(\eta)\geq 0\) for all \(\eta\in \C^\infty_c(\O,\R^2)\),

\item[2)] if \(L>0\) is large enough then \(Q_h(\eta)<0\) for \(\eta=(\cos\frac{\pi x}{2L} \sin \frac{\pi y}{2L}, -\sin \frac{\pi x}{2L}\cos \frac{\pi y}{2L})^T\).
\end{itemize}
\end{proposition}

\begin{proof}
We can check by direct computation that \(-\Delta h+h=-2\mathcal{H}^1_{\lfloor \{x=0\}}\) since the 1D function satisfies \(-h''+h=-2\delta_{x=0}\). Besides the condition \eqref{eq:stress_energy} is equivalent to \( (|h'|^2-h^2)'=0\) in \((-L,L)\) since \(h\) is a function of one variable. But we have that \(|h'|^2=|h|^2\) so \eqref{eq:stress_energy} is satisfied. We now consider the stability/instability properties.
 \begin{itemize}
 \item[1)] We first observe that 
\begin{align}\label{eq:area_energy_inequality}
|\det \Deriv \eta|&=|\p_1 \eta \wedge \p_2\eta |\leq |\p_1 \eta||\p_2\eta| \leq \frac12 (|\p_1 \eta|^2+|\p_2 \eta|^2)=\frac{|\Deriv \eta|^2}{2}.
\end{align}
Hence 
\begin{align*}
Q_h(\eta)&\geq \int_\O \left[|\Deriv \eta^T\nabla^\perp h|^2- (|\nabla h|^2+h^2)\det \Deriv \eta+h^2(\dive \eta)^2 \right] \\
&\geq \int_\O \left[ |h'|^2|\nabla \eta_2|^2-(|h'|^2+h^2)\det \Deriv \eta+h^2 (\dive \eta)^2 \right].
\end{align*}
 Then we show the following Poincaré type inequality: for every \(\eta_1\in \C^\infty_c(\O,\R)\),
 \begin{equation}\label{eq:Poincaré_ineq}
 \int_{(-L,L)^2} e^{-2|x|}|\eta_1(x,y)|^2 \dd x \dd y \leq 2L(e^{2L}-1) \int_{ (-L,L)^2} e^{-2|x|} |\p_1 \eta_1(x,y)|^2 \dd x \dd y.
 \end{equation}
Indeed, we write 
\begin{align*}
\int_{(-L,L)^2} h^2 |\eta_1|^2 &=\int_{ (-L,L)^2} h^2(x)\left( \int_{-L}^x \p_1\eta_1(s,y) \dd s \right)^2 \dd x \dd y \\
&\leq \int_{ (-L,L)^2} h^2(x) \int_{-L}^x|\p_1\eta_1(s,y)|^2 \dd s (x+L) \dd x \dd y \\
&\leq 2L \int_{-L}^L h^2(x) \dd x \int_{(-L,L)^2} |\p_1\eta_1(s,y)|^2 \dd s \dd y \\
&\leq 2L\int_{-L}^L h^2(x) \dd x \times e^{2L} \times \int_{(-L,L)^2} e^{-2|s|}|\p_1\eta_1(s,y)|^2 \dd s \dd y \\
&\leq 2L (1-e^{-2L})e^{2L}  \int_{(-L,L)^2} e^{-2|s|}|\p_1\eta_1(s,y)|^2 \dd s \dd y.
\end{align*} 
We notice that, since \( |h'|^2=h^2=e^{-2|x|}\)
\begin{align*}
\int_\O \left( |h'|^2+h^2\right) \det \Deriv \eta &= \frac12 \int_\O \left( |h'|^2+h^2\right) \dive (\eta \wedge \p_2\eta, \p_1 \eta \wedge \eta) \\
&=-\frac12 \int_\O \left( |h'|^2+h^2\right)' \eta \wedge \p_2 \eta =-\int (h^2)'(\eta_1\p_2\eta_2-\eta_2\p_2\eta_1) \\
&=-2\int_\O (h^2)'\eta_1\p_2\eta_2=-4\int_\O h h' \eta_1\p_2\eta_2.
\end{align*}
By using successively two Young's inequalities, by observing that \( |h'|^2=|h^2|=e^{-2|x|}\) and by employing the former Poincar\'e's inequality \eqref{eq:Poincaré_ineq} we find that 
\begin{align*}
Q_h(\eta) &\geq \int_\O |h'|^2 |\nabla \eta_2|^2+4hh'\eta_1\p_2\eta_2+h^2(\p_1\eta_1+\p_2\eta_2)^2 \\
& \geq \int_{\O} |h'|^2 |\nabla \eta_2|^2-2\alpha^2h^2 |\eta_1|^2-2|h'|^2\frac{|\p_2\eta_2|^2}{\alpha^2} \\
& \quad +h^2 (|\p_1\eta_1|^2+|\p_2 \eta_2|^2-\beta^2|\p_1\eta_1|^2-\frac{|\p_2 \eta_2|^2}{\beta^2} \\
& \geq \int_\O |h'|^2 \Bigl[ |\p_2\eta_2|^2(2-\frac{2}{\alpha^2}-\frac{1}{\beta^2})+|\p_1 \eta_2|^2 \\
& \quad +|\p_1\eta_1|^2\left(1-\beta^2-4\alpha^2L(e^{2L}-1) \right)\Bigr].
\end{align*}
Now we choose first \(\beta\) so that \(1-\beta^2>0\) and \(2-\frac{1}{\beta^2}>0\). This amounts to take \( 1/\sqrt{2}<\beta<1\). Then we choose \(\alpha \) big enough so that \( 2-\frac{2}{\alpha^2}-\frac{1}{\beta^2}>0\) and it remains to adjust \(L\) to have \(1-\beta^2-4\alpha^2L(e^{2L}-1)>0\). Thus the first point is proved.
\medskip

\item[2)] Let \(\eta=(\cos\frac{\pi x}{2L} \sin \frac{\pi y}{2L}, -\sin \frac{\pi x}{2L}\cos \frac{\pi y}{2L})^T\),  we can compute that 
\[ \Deriv \eta =\frac{\pi}{2L}\begin{pmatrix}
-\sin \frac{\pi x}{2L}\sin \frac{\pi y}{2L} & \cos \frac{\pi x}{2L}\cos \frac{\pi y}{2L} \\
-\cos \frac{\pi x}{2L}\cos \frac{\pi y}{2L} & \sin \frac{\pi x}{2L}\sin \frac{\pi y}{2L}.
\end{pmatrix}\]
Thus \( \frac{|\Deriv \eta|^2}{2}=\frac{\pi^2}{4L^2}(\sin^2 \frac{\pi x}{2L}\sin^2 \frac{\pi y}{2L}+\cos^2 \frac{\pi x}{2L}\cos^2 \frac{\pi y}{2L})\) and \(\det \Deriv \eta= \frac{\pi^2}{4L}(-\sin^2 \frac{\pi x}{2L}\sin^2 \frac{\pi y}{2l}+\cos^2 \frac{\pi x}{2L}\cos^2 \frac{\pi y}{2L})\). Thus we see that
\begin{align*}
\int_\O \left(\frac{|\Deriv \eta|^2}{2}-\det \Deriv \eta \right) \dd |\mu|= \frac{\pi^2}{4L^2}\int_{-L}^L 2\sin^2(0) \sin^2\frac{\pi y}{2L} \dd y =0.
\end{align*}

On the other hand, direct computations show that
\begin{align*}
\int_\O |h'|^2|\nabla \eta_2|^2&= \frac{\pi^2}{4L^2}\int_{(-L,L)^2} |h'|^2(\cos^2\frac{\pi x}{2L}\cos^2 \frac{\pi y}{2L}+\sin^2\frac{\pi x}{2L}\sin^2 \frac{\pi y}{2L}) \\
&= \frac{\pi^2}{4L^2}\times L \times \int_{-L}^L e^{-2|x|}  =\frac{\pi^2 (1-e^{-2L})}{4L},
\end{align*}
and 
\begin{align*}
\int_\O (|h'|^2+h^2)\det \Deriv \eta= &\frac{\pi^2}{4L}\int_{-L}^L 2e^{-2|x|} \left( \cos^2\frac{\pi x}{2L}-\sin^2\frac{\pi x}{2 L}\right) =\frac{\pi^2}{2L} \int_{-L}^L e^{-2|x|}\cos \frac{\pi x}{L}\\
&=\frac{\pi^2}{L} \int_{0}^L e^{-2x}\cos \frac{\pi x}{L}=\frac{\pi^2}{L}\RE \int_0^L e^{-2x+\frac{i \pi x}{L}} \\ &=\frac{2\pi^2 L^2 (1+e^{-2L})}{(4L^2+\pi^2)L} =\frac{2\pi^2 L (1+e^{-2L})}{(4L^2+\pi^2)}.
\end{align*}
We also observe that \(\dive \eta=0\). Hence
\begin{align*}
Q_h(\eta)&= \frac{\pi^2 (1-e^{-2L})}{4L}-\frac{2\pi^2L (1+e^{-2L})}{(4L^2+\pi^2)} \\
&= \frac{\pi^2}{4L(4L^2+\pi^2)}\left[ (4L^2+\pi^2)(1-e^{-2L})-8L^2 (1+e^{-2L}) \right] \\
&= \frac{\pi^2}{4L(4L^2+\pi^2)} \left[ -4L^2+\pi^2-e^{-2L}(12L^2+\pi^2)\right].
\end{align*}
It is easily seen that when \(L\) is large enough this quantity is negative.
 \end{itemize}
\end{proof} 
\subsection{The case without magnetic field}

In the case without magnetic field, the limiting stability condition never implies any further regularity on the limiting vorticity measure \(\mu\).

\begin{proposition}\label{prop:No_regularity_without}
 Let \(\mu\) be in \(H^{-1}(\O)\) and \(U\) be in \(H^1(\O)\) satisfying \eqref{eq:critical_conditions_without}, then \(\tilde{Q}_U(\eta)\geq 0\) for every \(\eta\) in \(\C^\infty_c(\O,\R^2)\); with \(\tilde{Q}_U\) defined in \eqref{eq:stability_limit_without}.
\end{proposition}

\begin{proof}
By using \eqref{eq:area_energy_inequality} we find that for every \(\eta \in \C^\infty_c(\O,\R^2)\),
\begin{align*}
\tilde{Q}_U(\eta)&\geq \int_\O |\Deriv \eta^T\nabla^\perp U|^2-|\nabla U|^2 \det \Deriv \eta \\
& =\int_\O \nabla^\perp U \otimes \nabla^\perp U: \Deriv \eta \Deriv \eta^T -|\nabla U|^2 \det \Deriv \eta \\
&=\int_\O \left(\nabla^\perp U \otimes \nabla^\perp U-\frac{|\nabla U|^2}{2}\Id\right): \Deriv \eta \Deriv \eta^T  +\int_\O |\nabla U|^2 \left(\frac{|\Deriv \eta|^2}{2}- \det \Deriv \eta \right) \\
&=\int_\O \left(\nabla^\perp U \otimes \nabla^\perp U-\frac{|\nabla U|^2}{2}\Id\right) :\left(\Deriv \eta \Deriv \eta^T-\frac{|\Deriv \eta|^2 }{2}\Id\right) \\
& \quad \quad +\int_\O |\nabla U|^2 \left(\frac{|\Deriv \eta|^2 }{2}- \det \Deriv \eta \right).
\end{align*}
In the last equality we have used that \( \left(\nabla^\perp U \otimes \nabla^\perp U-\frac{|\nabla U|^2}{2}\Id\right):\Id=\tr (\nabla^\perp U \otimes \nabla^\perp U-\frac{|\nabla U|^2}{2}\Id)=0\).
Now we remark that \((\nabla^\perp U \otimes \nabla^\perp U-\frac{|\nabla U|^2}{2}\Id)=-\left( \nabla U\otimes \nabla U-\frac{|\nabla U|^2}{2}\Id\right)\).

We take advantage of the complex structure of \(\R^2\simeq \mathbb{C}\) and, by denoting \(\p_z=(\p_1-i\p_2)/2\) and \(\p_{\bar{z}}=(\p_1+i\p_2)/2\), we can prove that 
\begin{equation}\label{eq:1st_complex_to_matrix}
\int_\O \left(\nabla^\perp U \otimes \nabla^\perp U-\frac{|\nabla U|^2 }{2}\Id\right) :\left(\Deriv \eta  \Deriv \eta^T-\frac{|\Deriv \eta|^2 }{2}\Id\right) =8\RE \int_\O \p_zU \overline {(\p_{\bar{z}}U)} \p_z \eta \p_{\bar{z}} \eta
\end{equation}
\begin{equation}\label{2nd_complex_to_matrix}
\int_\O |\nabla U|^2\left(\frac{|\Deriv \eta|^2 }{2}- \det \Deriv \eta \right)=4 \int_\O (|\p_z U|^2+|\p_{\bar{z}}U|^2)|\p_{\bar{z}}\eta|^2.
\end{equation}
Indeed, on the one hand
\begin{align*}
\left(\nabla^\perp U \otimes \nabla^\perp U-\frac{|\nabla U|^2 }{2}\Id\right) :\left(\Deriv \eta  \Deriv \eta^T-\frac{|\Deriv \eta|^2 }{2}\Id\right) \\
= \begin{pmatrix}
\frac{|\p_2U|^2-|\p_1U|^2}{2} & -\p_2U\p_1U \\
-\p_2U\p_1U & \frac{|\p_1U|^2-|\p_2U|^2}{2}
\end{pmatrix}: \begin{pmatrix}
\frac{|\nabla \eta_1|^2-|\nabla \eta_2|^2}{2} & \nabla \eta_1\cdot \nabla \eta_2 \\
\nabla \eta_1\cdot \nabla \eta_2 & \frac{|\nabla \eta_2|^2-|\nabla \eta_1|^2}{2}
\end{pmatrix}\\
= \frac12 (|\p_2U|^2-|\p_1U|^2)(|\nabla \eta_1|^2-|\nabla \eta_2|^2)-2 \p_1U \p_2U \nabla \eta_1\cdot \nabla \eta_2
\end{align*}
and on the other hand
\begin{align*}
&16\RE \left( \p_zU \overline{(\p_{\bar{z}} U)} \p_z\eta\p_{\bar{z}}\eta \right) \\
&=\RE \Bigl\{ (\p_1U-i\p_2U) (\p_1U-i\p_2U) \left(\p_1(\eta_1+i\eta_2)-i\p_2(\eta_1+i\eta_2)\right) \left(\p_1(\eta_1+i\eta_2)+i\p_2(\eta_1+i\eta_2)\right)\Bigr\}\\
&=\RE \Bigl\{ (|\p_1U|^2-|\p_2U|^2-2i\p_1U\p_2U) \Bigl[ (\p_1\eta_1+\p_2\eta_2)(\p_1\eta_1-\p_2\eta_2) -(\p_1\eta_2-\p_2\eta_1)(\p_1\eta_2+\p_2\eta_1) \\
& \quad \quad +i\Bigl( (\p_1\eta_2-\p_2\eta_1)(\p_1\eta_1-\p_2\eta_2)+(\p_1\eta_1+\p_2\eta_2)(\p_1\eta_2+\p_2\eta_1) \Bigr) \Bigr] \Bigr\} \\
&=\RE \Bigl\{ \left(|\p_1U|^2-|\p_2U|^2-2i\p_1U\p_2U \right) \left( |\nabla \eta_1|^2-|\nabla \eta_2|^2+2i \nabla \eta_1\cdot \nabla \eta_2 \right) \Bigr\} \\
&=(|\p_2U|^2-|\p_1U|^2)(|\nabla \eta_1|^2-|\nabla \eta_2|^2)-4 \p_1U \p_2U \nabla \eta_1\cdot \nabla \eta_2.
\end{align*}
This proves \eqref{eq:1st_complex_to_matrix} and \eqref{2nd_complex_to_matrix} is proved in a similar way.

We are thus led to prove that 
\begin{equation}\label{eq:true}
\frac12 \int_\O (|\p_z U|^2+|\p_{\bar{z}}U|^2)|\p_{\bar{z}}\eta|^2-\RE \int_\O \p_zU \overline {(\p_{\bar{z}}U)} \p_z\eta \p_{\bar{z}} \eta \geq 0 \quad \forall \eta \in \C^\infty_c(\O,\mathbb{C}).
\end{equation}
Since \(U\) is real-valued it satisfies that \((\p_zU)^2=(\p_zU)\overline{(\p_{\bar{z}}U)}\) and from equation \eqref{eq:critical_conditions_without} (see also \cite[Theorem 3]{Sandier_Serfaty_2003} or \cite[Theorem 13.2]{Sandier_Serfaty_2007}) we know that \((\p_zU)^2\) is holomorphic in \(\O\). We can invoke Theorem 1.10 in \cite{Iwaniec_Onninen_2022} to conclude that \eqref{eq:true} is true. Note that  in the statement of  Theorem 1.10 in \cite{Iwaniec_Onninen_2022} the quantity appearing is 
\[ \frac12 \int_\O (|\p_z U|^2+|\p_{\bar{z}}U|^2)|\p_{\bar{z}}\eta|^2+\RE \int_\O \p_zU \overline{(\p_{\bar{z}}U)} \p_z \eta \p_{\bar{z}} \eta.\]
But the proof of the non-negativity of this quantity for \(U\) such that \( (\p_zU)\overline{(\p_{\bar{z}}U)}\) is holomorphic  and for all \(\eta \in \C^\infty_c(\O,\mathbb{C})\) adapts with the minus sign, i.e.\ for the quantity appearing in \eqref{eq:true}. Indeed the proof of this fact rests upon the inequality
\[ \int_\O (\p_zU)\overline{(\p_{\bar{z}}U)} |\p_{\bar{z}}\eta|^2 \geq \left| \int_\O (\p_zU)\overline{(\p_{\bar{z}} U)}\p_z\eta \p_{\bar{z}} \eta\right|\]
valid for \(U\) satisfying that \( (\p_zU)\overline{(\p_{\bar{z}}U)}\) is holomorphic and for all \(\eta \in \C^\infty_c(\O,\mathbb{C})\), cf.\ Lemma 1.11 in \cite{Iwaniec_Onninen_2022}, and then we use 
\[ -\RE \int_\O \p_zU \overline{(\p_{\bar{z}}U)} \p_z \eta \p_{\bar{z}} \eta \geq -\left| \int_\O (\p_zU)\overline{(\p_{\bar{z}}U)} \p_z\eta \p_{\bar{z}} \eta\right|\] instead of 
\[ \RE \int_\O \p_zU \overline{(\p_{\bar{z}}U)} \p_z \eta \p_{\bar{z}} \eta \geq -\left| \int_\O (\p_zU)\overline{(\p_{\bar{z}}U)} \p_z\eta \p_{\bar{z}} \eta\right|\]  in the proof of Theorem 1.10 in \cite{Iwaniec_Onninen_2022}  to arrive at \eqref{eq:true}.
\end{proof}

\section{Conclusion and perspectives}

We have shown, in a certain regime of applied magnetic field \eqref{eq:magnetic_field_bound} and for solutions satisfying the energy bound \eqref{eq:energy_bound}, how to pass to the limit in the second inner variations of the energy \(\GL_\e\) if we assume the convergence of energies \eqref{cond:convergence_energies}. Since the \(\Gamma\)- limit \(E^\lambda\) of the sequence of energies \(\GL_\e\) is convex whereas the energies \(\GL_\e\) are not convex it is not direct to guess a limiting criticality condition (respectively a limiting stability condition) for solutions to \eqref{eq:GL_equations_magnetic}, (respectively  stable solutions) to \eqref{GL_magnetic}. In particular whereas limiting vorticity measures of solutions to \eqref{eq:GL_equations_magnetic} satisfy \(-\Delta h+h=\mu\) in \(\O\) with \(h\) which is stationary (i.e.\ critical for the inner variations) for \(\mathcal{L}(h)=\int_\O (|\nabla h|^2+h^2)\) it is not true that stable limiting vorticities of stable solutions verify that the second inner variation of \(\mathcal{L}\) is non-negative since this second inner variation can be computed to be equal to \[\delta^2\mathcal{L}(h,\eta)= \delta \mathcal{L}(h,\Deriv \eta.\eta)+\int_\O \left( |\Deriv \eta^T\nabla h|^2-(|\nabla h|^2-h^2)\det \Deriv \eta\right).\] The right limiting stable condition is given by \eqref{eq:stability_condition_magnetic}-\eqref{eq:limiting_stability}. The example analysed in Section \ref{sec:Analysis_limiting_condition} tends to show that the stability condition does not prevent limiting vorticity measures to concentrate on curves and that no further regularity for stable limiting vorticity could be deduced. This is definitely the case for the GL equations without magnetic field as shown by Proposition \ref{prop:No_regularity_without}. 

As for \cite[Theorem 1]{Sandier_Serfaty_2003}, our result Theorem \ref{th:main1} is interesting only if the total number of vortices \(N_\e=\sum_{i=1}^{M_\e} |d_i^\e|\) appearing in \eqref{def:N_eps} is of the same order as \(h_{\ex}\). As explained in \cite[Theorem 2]{Sandier_Serfaty_2003}, for \(\{(u_\e,A_\e)\}_{\e>0}\)  a family of solutions to \eqref{eq:GL_equations_magnetic}, if \( N_\e \gg  h_{\ex}\) then \(\mu(u_\e,A_\e)/N_\e\) converges to zero in the sense of measures whereas if \( N_\e\ll h_{\ex}\) then \(\mu(u_\e,A_\e)/N_\e\rightharpoonup \mu\) with \(\mu \nabla h_0=0\) and \(h_0\) the solution to \(-\Delta h_0+h_0=0\) in \(\O\) with \(h=1\) on \(\p \O\) and hence the support of \(\mu\) is included in the set of critical points of \(h_0\). For minimizers, it was proved in \cite{Sandier_Serfaty_2003_Calvar,Sandier_Serfaty_2007} that vortices accumulate near minimizing points of \(h_0\). We can also ask if there exist supplementary conditions in the limit \(\e\to 0\) for stable solutions with \( N_\e\ll h_{\ex}\) such as vortices accumulating towards stable critical points of \(h_0\) in \(\O\). However this seems to require different techniques than the ones used in this paper.

\section*{Appendix}
Here we recall two results used in the proof of main theorems. These results aim at describing the limiting vorticities near regular points of the limiting field \(h\). Note that we can define regular and critical points of \(h\) since it is proved in \cite[Theorem 13.1]{Sandier_Serfaty_2007} that \(|\nabla h|^2\) is continuous in \(\O\).

\begin{theorem}(\cite[Theorem 3.1]{Rodiac_2019})
Let \(h\in H^1(\O)\) and \(\mu \in \mathcal{M}(\O)\) be such that \(-\Delta h+h=\mu\) and \(\sum_{j=1}^2 \p_j\left[ 2\p_i h\p_j h-\left( |\nabla h|^2+h^2\right) \delta_{ij}\right]=0\) in \(\O\) for \(i=1,2\). Let \(x_0\in \supp \mu\) be such that \(|\nabla h(x_0)|\neq 0\). Then there exists \(R>0\) and \(H\in \C^{1,\alpha}(B(x_0,R))\) for every \(0<\alpha <1\) such that 
\begin{equation*}
\supp \mu_{\lfloor B(x_0,R)}=\{x \in B(x_0,R): H(x)=0\}=:\Gamma
\end{equation*}
and \(\nabla H(x)\neq 0\) for every \(x\in B(x_0,R)\). Furthermore \(\mu_{\lfloor B(x_0,R)}=+2|\nabla h|\mathcal{H}^1_{\lfloor \Gamma}\) or \(\mu_{\lfloor B(x_0,R)}=-2|\nabla h|\mathcal{H}^1_{\lfloor \Gamma}\).
\end{theorem}

\begin{theorem}(\cite[Theorem 1.3]{Rodiac_2016})
Let \(h\in H^1(\O)\) and \(\mu \in \mathcal{M}(\O)\) be such that \(\Delta h=\mu\) and \(\sum_{j=1}^2 \p_j\left[ 2\p_i h\p_j h- |\nabla h|^2 \delta_{ij}\right]=0\) in \(\O\) for \(i=1,2\). Let \(x_0\in \supp \mu\) be such that \(|\nabla h(x_0)|\neq 0\). Then there exists \(R>0\) and \(H\) a harmonic function in \(B(x_0,R)\) such that 
\begin{equation*}
\supp \mu_{\lfloor B(x_0,R)}=\{x \in B(x_0,R): H(x)=0\}=:\tilde{\Gamma}
\end{equation*}
and \(\nabla H(x)\neq 0\) for every \(x\in B(x_0,R)\). Furthermore \(\mu_{\lfloor B(x_0,R)}=+2|\nabla h|\mathcal{H}^1_{\lfloor \tilde{\Gamma}}\) or \(\mu_{\lfloor B(x_0,R)}=-2|\nabla h|\mathcal{H}^1_{\lfloor \tilde{\Gamma}}\).
\end{theorem}

\bibliographystyle{abbrv}
\bibliography{bib}
\end{document}